\documentclass[12pt]{amsart}

\usepackage{amsfonts,amssymb,stmaryrd,amscd,amsmath,latexsym,amsbsy}

\newtheorem{theorem}{Theorem}[section]
\newtheorem{lemma}[theorem]{Lemma}
\newtheorem{proposition}[theorem]{Proposition}
\newtheorem{corollary}[theorem]{Corollary}
\theoremstyle{definition}
\newtheorem{definition}[theorem]{Definition}

\newtheorem{example}[theorem]{Example}
\newtheorem{question}[theorem]{Question}

\newtheorem{remark}[theorem]{Remark}

\newcommand{\Hom}{\text{Hom}}

\newcommand{\kk}{{\bold k}}

\newcommand{\C}{\mathcal{C}}

\newcommand{\ben}{\begin{enumerate}}
\newcommand{\een}{\end{enumerate}}

\theoremstyle{plain}

\newtheorem*{sol}{Solution}

\theoremstyle{definition}

\theoremstyle{remark}

\newcommand{\solu}[1]{\begin{sol}{\bf (\ref{#1})}}

\begin{document}

\title{Koszul duality and the PBW theorem in symmetric tensor categories in positive characteristic}

\author{Pavel Etingof}
\address{Department of Mathematics, Massachusetts Institute of Technology,
Cambridge, MA 02139, USA}
\email{etingof@math.mit.edu}

\dedicatory{To David Kazhdan on his 70th birthday with admiration}

\begin{abstract} We generalize the theory of Koszul complexes and Koszul algebras to symmetric tensor categories. 
In characteristic zero the generalization is routine, while in characteristic $p$ 
there is a subtlety -- the symmetric algebra of an object is not always Koszul
(i.e., its Koszul complex is not always exact). Namely, this happens in the Verlinde 
category ${\rm Ver}_p$ in any characteristic $p\ge 5$. We call an object Koszul 
if its symmetric algebra is Koszul, and show that the only Koszul objects 
of ${\rm Ver}_p$ are usual supervector spaces, i.e., a non-invertible 
simple object $L_m$ ($2\le m\le p-2$) is not Koszul. We show, however, that 
the symmetric algebra 
$SL_m$ is almost Koszul in the sense of Brenner, Butler
and King (namely, $(p-m,m)$-Koszul), and compute the corresponding internal 
Yoneda algebra (i.e., the internal Ext-algebra from the trivial module to itself). 

We then proceed to discuss the PBW theorem for operadic Lie algebras (i.e., algebras over the operad ${\bf Lie}$). 
This theorem is well known to fail for vector spaces in characteristic $2$ 
(as one needs to require that $[x,x]=0$), and for supervector spaces
in characteristic $3$ (as one needs to require that $[[x,x],x]=0$ 
for odd $x$), but it holds in these categories in any characteristic 
$p\ge 5$; there is a well known proof based on Koszul duality. 
However, we show that in the category ${\rm Ver}_p$, because of failure of Koszul duality, 
the PBW theorem can fail in any characteristic $p\ge 5$. Namely, one 
needs to impose the $p$-Jacobi identity, a certain generalization to 
characteristic $p$ of the identities $[x,x]=0$ and $[[x,x],x]=0$. 
On the other hand, our main result is that once the $p$-Jacobi identity is imposed, 
the PBW theorem holds. This shows that the correct definition of a 
Lie algebra in ${\rm Ver}_p$ is an algebra over ${\bf Lie}$ which satisfies the $p$-Jacobi identity. 
This also applies to any symmetric tensor category that admits a symmetric tensor functor to ${\rm Ver}_p$ 
(e.g., a symmetric fusion category, see \cite{O}, Theorem 1.5). Finally, we prove the PBW theorem for Lie algebras 
in any quasi-semisimple symmetric tensor category.       
\end{abstract}

\maketitle

\section{Introduction}

The goal of this paper is to extend the theory of Koszul duality 
and Koszul algebras to symmetric tensor categories, and to use it to
prove the PBW theorem for Lie algebras in such categories. 

We start by defining the Koszul complex $K^\bullet(V)$ of an object 
$V$ of a symmetric tensor category $\C$ over a field $\kk$
(following \cite{EHO}, Subsection 3.4). In the classical setting (of vector and supervector spaces), 
this complex is exact, which is a basic fact in commutative algebra. 
In particular, it is exact in the category of Schur functors, and hence 
in any symmetric tensor category if ${\rm char}(\kk)=0$. 

In characteristic $p$, however, the story is more complicated, and the Koszul complex may fail to be exact. 
Thus we call an object $V$ Koszul if its Koszul complex is exact. We then study the Koszul complex of an object 
$V$ of the Verlinde category ${\rm Ver}_p$ (\cite{O}, Definition 3.1) and show that the only $V$ for which $K^\bullet(V)$ is exact 
are classical supervector spaces. In other words, if $V=L_m$, where $L_m$ is a non-invertible simple object of ${\rm Ver}_p$
(i.e., $2\le m\le p-2$), then the symmetric algebra\footnote{Throughout the paper, this algebra should not be confused with the special linear group, which we will denote by ${\bold{SL}}(m)$.} $SL_m$ is not Koszul; this happens for each $p\ge 5$. 

We also generalize to the categorical setting the theory of Koszul algebras, 
and in particular Drinfeld's ``Koszul deformation principle'' stating that a 
homogeneous deformation of a Koszul algebra is flat if it is flat in degrees 
$2$ and $3$ (\cite{PP}). Also, we show that if $V$ is a Koszul object then the algebras $SV$ and $\wedge V^*$ are Koszul. 
Finally, we show that if $2\le m\le p-2$ then the algebras $SL_m$ and $\wedge L_m$ are 
almost Koszul in the sense of \cite{BBK}, and for any $V\in {\rm Ver}_p$, compute the internal Ext-algebra 
of the augmentation module over $SV$ and $\wedge V^*$ to itself, using the periodic 
Koszul complex of $L_m$ (which, unlike the usual Koszul complex, is exact). 
  
We then proceed to apply these results to the Poincare-Birkhoff-Witt theorem for Lie algebras in symmetric tensor categories. 
Again, in characteristic zero the generalization of the PBW theorem is straightforward, and several proofs 
extend easily to the categorical setting (e.g. the one based on the Campbell-Baker-Hausdorff formula, see \cite{EGNO}, Exercise 9.9.7(viii), 
or the one based on the Koszul deformation principle, given below). On the other hand, in characteristic $p$ 
it is not even clear what the right definition of a Lie algebra should be. Namely, the most obvious definition  
is that of an operadic Lie algebra, i.e. an algebra over the operad ${\bf Lie}$. 
But it is known that already for usual Lie algebras, this definition is not
correct in characteristic $2$: to ensure the PBW theorem, one needs to additionally impose 
the relation $[x,x]=0$, which is not polylinear and hence not of operadic nature. 
Furthermore, for Lie superalgebras, a similar problem occurs in characteristic $3$, 
and to have the PBW theorem, one needs to impose the relation $[[x,x],x]=0$ for odd $x$. 
With these additional axioms, however, in the category of supervector spaces 
the PBW theorem always holds, as can be shown using the Koszul deformation principle. 
A similar statement (with a suitable categorical generalization of the conditions $[x,x]=0$ and $[[x,x],x]=0$ and the same proof) 
holds in any symmetric tensor category if the underlying object of the operadic Lie algebra is Koszul.     

On the other hand, in absence of Koszulity, the situation is more complicated, and 
the PBW theorem for operadic Lie algebras can fail in any characteristic $p\ge 5$. 
So it remains unclear how to define a Lie algebra. To deal with this issue, for any
object $V$ we consider the free operadic Lie algebra ${\rm FOLie}(V)$ generated by $V$, which maps naturally to the tensor algebra $TV$. Unlike the classical situation, this map, in general, is not injective, and 
has a kernel $E(V)=\oplus_{n\ge 1}E_n(V)$ (where $E_n(V)$ is the degree $n$ component). 
The object $E(V)$ is crucial for understanding the failure of the PBW theorem, and we define 
a Lie algebra as an operadic Lie algebra $L$ for which the natural map 
$\beta^L: {\rm FOLie}(L)\to L$ annihilates $E(L)$. This always holds 
if $L$ satisfies the PBW theorem, and we show that the converse is true in the category ${\rm Ver}_p$ 
(and hence in every category that admits a fiber functor to ${\rm Ver}_p$, e.g. any symmetric fusion category, see \cite{O}, Theorem 1.5). 

More precisely, we show that in ${\rm Ver}_p$, we have 
$E_n(V)=0$ for $n<p$ and $E_p(V)={\rm Hom}(L_2,V)^{(1)}\otimes {\bold 1}_-$, 
where for $p=2$ we agree that $L_2={\bold 1}_-={\bold 1}$. Moreover, we show that 
the axiom of a Lie algebra follows from its specialization to degree $p$, which we call 
the $p$-Jacobi identity: $\beta^L|_{E_p(L)}=0$. Since $E_p(L)$ is known explicitly, this 
is an explicit identity of degree $p$ which needs to be added to 
the skew-symmetry and usual Jacobi identity to define a Lie algebra. 
And once this identity is imposed, our main result 
guarantees that the PBW theorem holds. 

We also show that ${\rm FOLie}(L_2)_{\le p}$ (the quotient of ${\rm FOLie}(L_2)$ by subobjects of degree 
$\ge p+1$) is a finite dimensional operadic Lie algebra in ${\rm Ver}_p$ which is not a 
Lie algebra (i.e., it fails the PBW theorem).  

In a general symmetric tensor category in characteristic $p$, we don't know if 
the PBW theorem holds for any Lie algebra, but we prove it 
under a mild restriction (that the category is quasi-semisimple, i.e., any injection $X\to Y$ defines an injection $SX\to SY$), see Section 7. 
We also don't know if the $p$-Jacobi identity for an operadic Lie algebra $L$ implies that 
it satisfies the PBW theorem, or at least is a Lie algebra.   

The organization of the paper is as follows. In Section 2, we study the 
Koszul and De Rham complexes of an object, and introduce the notion of a Koszul object. In Section 3, we extend the theory of Koszul algebras and the Koszul deformation principle to tensor categories. In Section 4 we define the notion of a Lie algebra in a symmetric tensor category, and study its properties. We also prove
the PBW theorem in the case when the underlying object of an operadic Lie algebra is Koszul. In Section 5, we 
compute the cohomology of the Koszul and De Rham complexes of an object of ${\rm Ver}_p$, 
as well as the Ext-algebra of the augmentation module over the 
symmetric and exterior algebra of any object of ${\rm Ver}_p$,
showing that the algebras $SL_m$, $\wedge L_m$ for $2\le m\le p-2$ are almost Koszul. We also sketch a construction of Frobenius $(r,s)$-Koszul algebras for any $r,s\ge 2$ which arise from this approach. In Section 6 we prove our main result -- the PBW theorem 
in ${\rm Ver}_p$ in presence of the $p$-Jacobi identity, and discuss operadic Lie algebras in ${\rm Ver}_p$ 
that fail the PBW theorem. Finally, in Section 7 we prove the PBW theorem for Lie algebras in a quasi-semisimple
symmetric tensor category.  

{\bf Acknowledgements.} The author is grateful to D. Kazhdan, V. Ostrik, 
A. Polishchuk, and L. Positselski for very useful discussions, and to V. Ostrik for reading the draft and many interesting suggestions. In particular, the author thanks A. Polishchuk for 
help with the proof of Theorem \ref{maint}. 
The work of the author was partially supported by the NSF grant DMS-1502244. 

\section{The Koszul complex of an object} 

\subsection{The Koszul complex} Throughout the paper, $\kk$ will denote an algebraically closed field, and $\C$ a symmetric 
tensor category over $\kk$, i.e. an abelian $\kk$-linear category with a compatible rigid symmetric monoidal structure 
(see \cite{EGNO}, Definitions 4.1.1, 8.1.2). Let $V\in \C$. Recall from \cite{EHO}, Subsection 2.1 
that to $V$ we may attach objects $S^iV$, $\Bbb S^iV$, $\wedge^iV$, $\Lambda^iV$. 
Using these objects, we can define the symmetric algebra $SV:=\oplus_{i\ge 0}S^iV$, the dual symmetric algebra 
$\Bbb SV:=\oplus_{i\ge 0}\Bbb S^iV$, the exterior algebra 
$\wedge V:=\oplus_{i\ge 0}\wedge^i V$, and the dual exterior algebra 
$\Lambda V:=\oplus_{i\ge 0}\Lambda^i V$ (ind-objects\footnote{In this paper, we will often deal with ind-objects and sometimes pro-objects. To simplify the language, we will often abuse terminology by referring to them as ``objects".} in $\C$). These are $\Bbb Z_+$-graded 
algebras, and $SV$, $\wedge V$ are generated in degree $1$ (by $V$), while 
$\Bbb SV$, $\Lambda V$, in general, are not (if ${\rm char}(\kk)>0$). 
 
Note that $SV,\wedge V$ are graded Hopf algebras, with coproduct defined by $\Delta|_V=\Delta_1+\Delta_2$
where $\Delta_1: V\to V\otimes \bold 1$ and $\Delta_2: V\to \bold 1\otimes V$ are the unit morphisms.  
Thus, the restricted duals $(SV)^*_{\rm res}=\Bbb SV^*$ and $(\wedge V)^*_{\rm res}=\Lambda V^*$ are also graded Hopf algebras. 
This gives another description of the multiplication in $\Bbb SV$, $\Lambda V$. 

Recall that we have the multiplication morphism $\mu: V\otimes S^jV\to S^{j+1}V$.  
We also have the multiplication morphism $V^*\otimes \wedge^{i-1} V^*\to \wedge^i V^*$, 
which after dualization defines a morphism $\mu_*: V^*\otimes \Lambda^i V\to \Lambda^{i-1}V$.

\begin{definition} (\cite{EHO}, Subsection 3.4) The {\it Koszul complex} of $V$ is the complex $K^\bullet(V)=SV\otimes \Lambda^\bullet V$ 
whose differential $\partial : K^\bullet(V)\to K^{\bullet-1}(V)$ is
the composition 
$$
SV\otimes \Lambda V\to V\otimes V^*\otimes SV\otimes \Lambda V\to V\otimes SV\otimes V^*\otimes \Lambda V\to SV\otimes \Lambda V
$$
given by the formula 
$$
\partial=(\mu\otimes \mu_*)\circ P_{23}\circ ({\rm coev}_V\otimes 1_{SV}\otimes 1_{\Lambda V}),
$$
where ${\rm coev}_V: \bold 1\to V\otimes V^*$ is the coevaluation, and $P_{23}$ is the permutation of the second and the third factor. 
\end{definition}  

Thus, $\partial: S^jV\otimes \Lambda^i V\to S^{j+1}V\otimes \Lambda^{i-1}V$. In other words, 
$\partial $ preserves the diagonal grading on $K^\bullet(V)$ (i.e. the grading induced by putting both copies of $V$ in degree $1$), but decreases 
the homological grading $\bullet$ by $1$. Also, it is easy to see that 
$\partial$ is a derivation of the algebra $SV\otimes \Lambda V$, so 
the Koszul complex is a DG algebra. 

\begin{definition}\label{exaa} We will say that the Koszul complex is {\it exact} if it 
is exact in positive homological degrees and has homology equal to $\bold 1$ in degree $0$; i.e., if it is a resolution of the augmentation $SV$-module $\bold 1$ by free $SV$-modules.
\end{definition} 

As explained in \cite{EHO}, Subsection 3.4, in the category of vector spaces, the Koszul complex is exact, and coincides with the usual Koszul complex 
from commutative algebra, --- the differential graded algebra $SV\otimes \wedge^\bullet V$ with differential given by $\partial(x,y)=(0,x)$, $x,y\in V$. 
Moreover, it is exact in the category of supervector spaces. This implies that it is exact in classical representation categories 
(of supergroups) and their interpolations (Deligne categories). Therefore, if ${\rm char}(\kk)=0$ 
then the Koszul complex is exact in the category of Schur functors, and hence in any symmetric tensor category over $\kk$. 
For the same reason, it is exact in degrees $<p$ (under diagonal grading) 
in any symmetric tensor category over a field of characteristic $p$. 
However, it follows from Proposition 3.10 of \cite{EHO} that in general the Koszul complex 
can fail to be exact\footnote{We don't know if this can happen in characteristics $2,3$.}; 
in particular, this happens in the Verlinde categories ${\rm Ver}_p$ for any $p\ge 5$. 

\subsection{Internal Ext}

Let $A$ be an ind-algebra in $\C$, and $N\in \C$ an $A$-module. 
Then we have a functor of internal Hom into this module, 
$\underline{\rm Hom}_A(?,N)$. Namely, 
if $A\otimes Y\to A\otimes X\to M\to 0$ is a presentation 
of $M$, then $\underline{\rm Hom}_A(M,N)$ is the kernel 
of the map $X^*\otimes N\to Y^*\otimes N$ induced by the composite 
morphism $Y\to A\otimes X\to N\otimes N^*\otimes X$
 (note that in general $X,Y$ are ind-objects and $X^*,Y^*$ are pro-objects of $\C$).  
  
The functor $\underline{\rm Hom}_A(?,N)$ is left exact, and thus we can define 
its right derived functors $\underline{\rm Ext}^i_A(?,N)$, which  take values in pro-objects of $\C$ (even if $M\in \C$) and may be computed using a resolution of $M$ by free $A$-modules. Moreover, $\underline{\rm Ext}_A^\bullet(N,N)$ is an algebra under the 
Yoneda product. We will call the functors $\underline{\rm Ext}^i_A$ the {\it internal Ext functors}. 

\subsection{Koszul objects} 

\begin{definition}
An object $V\in \C$ is {\it Koszul} if the complex $K^\bullet(V)$ is exact. 
\end{definition}

\begin{proposition}\label{kosdual} If $V\in \C$ is Koszul then 
we have canonical graded algebra isomorphisms $\underline{\rm Ext}_{SV}^\bullet(\bold 1,\bold 1)=\wedge^\bullet V^*$, 
$\underline{\rm Ext}_{\wedge V^*}^\bullet(\bold 1,\bold 1)=S^\bullet V$, where $\bold 1$ is the augmentation module. 
Here the diagonal grading coincides with the cohomological grading. In other words, the internal Yoneda algebra 
of $SV$ is $\wedge V^*$, and vice versa. 
\end{proposition} 

\begin{proof} The proof is the same as in the classical case. Namely, we will 
use $K^\bullet(V)$ as a free resolution of $\bold 1$ as an $SV$-module.
Computing its internal Hom into $\bold 1$, we get the desired isomorphism 
(and it is checked as usual that it is compatible with products). 
Similarly, the restricted dual of $K^\bullet(V)$ may be used as 
a projective resolution of $\bold 1$ over $\wedge V^*$, 
and the argument works the same way.  
\end{proof} 

\subsection{The De Rham complex} \label{drc}

Another useful complex that we can attach to an object $V\in \C$ is the {\it De Rham} complex, the differential graded algebra 
$DR^\bullet(V)=SV\otimes \Lambda^\bullet V$ (i.e., same as $K^\bullet(V)$ as a graded algebra)
whose differential $d : DR^\bullet(V)\to DR^{\bullet+1}(V)$ is
the composition 
$$
SV\otimes \Lambda V\to V\otimes V^*\otimes SV\otimes \Lambda V\to V^*\otimes SV\otimes V\otimes \Lambda V\to SV\otimes \Lambda V
$$
given by the formula 
$$
d=(\bar\mu_*\otimes \bar\mu)\circ P_{1,23}\circ ({\rm coev}_V\otimes 1_{SV}\otimes 1_{\Lambda V}),
$$
where $P_{1,23}$ is the permutation of the first component with the product of the second and third one, 
$\bar\mu: V\otimes \Lambda V\to \Lambda V$ is the multiplication map of $\Lambda V$, and $\bar\mu_*: V^*\otimes SV\to SV$ is the ``differentiation" map obtained by dualizing the 
multiplication map of $\Bbb SV^*$. 

For example, if $V$ is a vector space, then 
$\Lambda V=\wedge V$, and we can define $DR^\bullet(V)$ as 
the differential graded algebra $SV\otimes \wedge^\bullet V$ with differential 
defined on generators by the formula $d(x,y)=(y,0)$, $x,y\in V$. 
In other words, $DR^m(V)$ is the space of differential $m$-forms
on $V^*$, and $d$ is the usual exterior differential.  

In characteristic zero, we have an isomorphism 
$DR^\bullet(V)_{\rm res}^*=K^\bullet(V^*)$, which in particular implies that $DR^\bullet(V)$ is exact in the sense of Definition \ref{exaa} 
(the Poincar\'e lemma). 
However, this is, in general, not true (and the Poincar\'e lemma fails) in positive characteristic. 
For instance, if $V$ is a vector space, then we have following well known result
(which one can prove for 1-dimensional spaces and then take the tensor product): 

\begin{proposition}\label{cartier} For any vector space $V$ over $\kk$ of characteristic $p$,
we have a canonical isomorphism  
$H^*(DR^\bullet(V))\cong SV^{(1)}\otimes \wedge V^{(1)}$,
where $V^{(1)}$ is the Frobenius twist of $V$ (called the Cartier isomorphism). 
Under this isomorphism, elements of diagonal degree $d$ 
on the RHS go to those of degree $pd$ on the LHS. 
\end{proposition} 

Also, we have the following result. 
Let ${\rm char}(\kk)=p$. Recall that an additive functor $F$ between $\kk$-linear categories 
is called twisted $\kk$-linear if $F(\lambda \alpha)=\lambda^p F(\alpha)$ for a morphism $\alpha$ and a scalar $\lambda \in \kk$.  

\begin{proposition}\label{divi} (i) For any symmetric tensor category $\C$ over $\kk$ and any $V\in \C$, 
the De Rham cohomology $H^*(DR^\bullet(V))$ lives in diagonal degrees 
divisible by $p$. 

(ii) The functor $V\mapsto H^*(DR^\bullet(V))[p]$ (diagonal degree $p$) 
is twisted $\kk$-linear.  
\end{proposition}

\begin{proof} Let $E$ be the Euler derivation of $DR^\bullet(V)$, i.e., the degree endomorphism.
Then by Cartan's formula, $E=d\circ \partial +\partial \circ d$, hence $E=d\circ \partial$ on ${\rm Ker}d$
(indeed, the Koszul differential $\partial$ coincides with the contraction endomorphism
$i_E$ associated to $E$). This implies the statement, since $E$ acts by a nonzero scalar in any degree not divisible by $p$. 

(ii) This follows from (i) and the K\"unneth formula. 
\end{proof} 

Note that in characteristic zero, this argument gives another proof of the Poincar\'e lemma (which 
replicates the standard proof in multivariable calculus). 

\section{Koszul algebras in symmetric tensor categories} 

As in the classical setting, the duality between $SV$ and $\wedge V^*$ discussed in the previous section is an instance of a more general phenomenon 
-- Koszul duality for associative algebras in symmetric tensor categories. The classical theory of Koszul algebras (see \cite{PP}) 
extends to the setting of symmetric tensor categories in a straightforward manner; so we will just state the results omitting most proofs (which are repetitions of classical proofs). 

Let $A$ be a $\Bbb Z_+$-graded algebra in a symmetric tensor category $\C$ with $A[0]=\bold 1$.
Namely, $A$ is a graded ind-object of $\C$, but for each $i$, the homogeneous part $A[i]$ of degree $i$ is an actual object of $\C$.  

\begin{definition} We say that $A$ is {\it quadratic} if $\underline{\rm Ext}^i_A(\bold 1,\bold 1)$ sits in degree\footnote{The natural grading for the Ext algebra is by negative integers, but we reverse the sign for convenience.} $i$ for $i=1,2$, and {\it Koszul} if this condition holds for all $i\ge 1$.  
\end{definition} 

As usual, $A$ is quadratic if and only if it is generated in degree $1$ with defining relations in degree $2$.  

Given a quadratic algebra $A$, we can define the quadratic dual $A^!$, the quadratic algebra generated by $A^![1]=A[1]^*$ 
with relations $R^\perp$, where $R$ is the object of relations of $A$, i.e., 
the kernel of the multiplication map $A[1]\otimes A[1]\to A[2]$. Then $(A^!)^!=A$. 
For example, $(SV)^!=\wedge V^*$ and $(\wedge V)^!=SV^*$. 

Also, for a quadratic $A$ we can define the {\it Koszul} complex $K^\bullet_A$ of $A$ by $K^i_A:=A\otimes A^![i]^*$,
with differential defined by the formula
$$
\partial=(\mu\otimes \mu_*)\circ P_{23}\circ ({\rm coev}_{A[1]}\otimes 1_{A}\otimes 1_{A^!}),
$$
where $\mu: A[1]\otimes A\to A$ is the multiplication map, and 
$$
\mu_*: A^![1]\otimes A^![i]^*\to A^![i-1]^*
$$ 
is obtained by dualizing the multiplication map of $A^!$. 

\begin{proposition}\label{kosdual1} 
(i) $A$ is Koszul if and only if the Koszul complex $K^\bullet_A$ is exact. 
In this case, $A^!$ is also Koszul and $\underline{\rm Ext}_{A^!}^\bullet(\bold 1,\bold 1)=A$. 

(ii) $A$ is Koszul if and only if the augmentation module $\bold 1$
admits a resolution $F^\bullet$ by free $A$-modules such that $F^i$ 
is generated in degree $i$. 
\end{proposition} 

\begin{proof} (i) Suppose $K^\bullet_A$ is exact. Then $K^\bullet_A$ 
is a resolution of $\bold 1$, which may be used to compute 
$\underline{\rm Ext}_A^\bullet(\bold 1,\bold 1)$. 
We have $\underline{\rm Hom}(K^\bullet_A,\bold 1)=A^!$, and it is clear that the differential is zero, since each term of the complex is in a different degree, while the differential has degree zero. Thus, $\underline{\rm Ext}_A^\bullet(\bold 1,\bold 1)=A^!$, and it is checked as in the classical case that this is an isomorphism of algebras. In particular, $A$ is Koszul. 
This also implies that $A^!$ is Koszul, since $K^\bullet_{A^!}$ is the 
restricted dual of $K^\bullet_A$. 

Conversely, if $A$ is Koszul then it is shown by induction as in the classical case that 
$K^\bullet_A$ is the minimal resolution of $\bold 1$ by free $A$-modules; in particular, it is exact. 

(ii) If $A$ admits such a resolution, then by taking internal Hom into $\bold 1$ we see that 
$A$ is Koszul. Conversely, if $A$ is Koszul then we can set $F^\bullet=K^\bullet_A$.
\end{proof} 

Thus, we obtain 

\begin{proposition}\label{koszobj} $V\in \C$ is Koszul if and only if either of 
the algebras $SV$ and $\wedge V^*$ is Koszul. 
\end{proposition}

Analogously to the classical case, we have Drinfeld's ``Koszul deformation principle" for 
Koszul algebras. 

\begin{theorem}\label{KdP} Let $A$ be a graded quadratic deformation over $\kk[[\hbar]]$ of a Koszul algebra $\bar A$ in a symmetric tensor category $\C$ (i.e., we formally deform the generators and the relations). Then $A$ is a flat deformation if it is flat in degrees 
$\le 3$. 
\end{theorem} 

\begin{proof} The proof is the same as in the classical case, see \cite{PP}, Section 6; namely, the theory
of distributive lattices extends to the setting of symmetric tensor categories.   
\end{proof} 

We note that Theorem \ref{KdP} applies to the situation when $A[1]$ is a nontrivial flat deformation of $\bar A[1]$. 

The following corollary of Theorem \ref{KdP} answers Question 3.5 in \cite{EHO} in the Koszul case (for $p\ge 5$). 

\begin{corollary}\label{koszulfil} Let $\C$ be a symmetric tensor category over a field $\kk$ 
of characteristic $p\ge 5$, and $Y$ is a filtered object of $\C$. If the object 
${\rm gr}Y$ is Koszul, then 

(i) the natural surjections 
$$
\phi_+: S({\rm gr}Y)\to {\rm gr}(SY),\ 
\phi_-: \wedge ({\rm gr}Y)\to {\rm gr}(\wedge Y)
$$
are isomorphisms;

(ii) $Y$ is a Koszul object.  
\end{corollary}

\begin{proof} (i) The proof mimicks the proof of Proposition 3.6 of \cite{EHO}. 
We prove the statement for $\phi_+$; the proof for $\phi_-$ is similar.

Let ${\rm Rees}(Y):=\prod_{j\ge 0}\hbar^j F_jY$. Then $S({\rm Rees}(Y))$ is a formal deformation of $S({\rm gr}Y)$, and (i) is equivalent to the statement that this deformation 
is flat. 

It is clear that this deformation is flat (i.e., $\phi_+$ is an isomorphism) in degrees $\le 3$, since
$S^2W=\bold e_2W^{\otimes 2}$ and $S^3W=\bold e_3W^{\otimes 3}$ for any $W\in {\mathcal C}$ 
(as ${\rm char}(\kk)\ne 2,3$). Hence, by Theorem \ref{KdP}, 
it is flat in all degrees, which yields (i). 

(ii) Dualizing $\phi_-$, we get a natural injection 
$\phi_-^*: {\rm gr}(\Lambda Y)\to \Lambda({\rm gr}Y)$, which is an isomorphism by (i). 
Hence, $K^\bullet({\rm gr}Y)\cong {\rm gr}K^\bullet(Y)$. Since $K^\bullet({\rm gr}Y)$ is exact, this implies 
that $K^\bullet(Y)$ is exact, hence $Y$ is Koszul.  
\end{proof} 

\begin{remark} Let $R\in {\rm End}(Y\otimes Y)$ be  a unitary solution of the quantum Yang-Baxter equation, i.e., 
$$
R^{21}R={\rm Id},\ R^{12}R^{13}R^{23}=R^{23}R^{13}R^{12}.
$$
Assume that $R$ preserves the filtration on $Y\otimes Y$, and ${\rm gr}R={\rm Id}$. 
Let $S^2_RY$ and $\wedge^2_RY$ be the invariants and anti-invariants of $P\circ R$, where 
$P\in {\rm End}(Y\otimes Y)$ is the permutation map, and let $S_RY:=TY/\langle \wedge^2_RY\rangle$, 
$\wedge_RY:=TY/\langle S^2_RY\rangle$ (so $S_RY=SY$ and $\wedge_R Y=\wedge Y$ if $R={\rm Id}$). Then Corollary 
\ref{koszulfil}(i) remains true (with the same proof) if $SV,\wedge V$ are replaced 
by $S_RV$, $\wedge_R V$. 
\end{remark} 

\begin{remark}\label{infii} We have outlined the theory of Koszul duality in the case when 
$A[i]$ are``finite dimensional", i.e., objects of $\C$; e.g., if $\C={\rm Vec}$, 
it covers only finitely generated algebras. However, as in the classical case, this theory 
admits a straightforward generalization to the case when $A[i]$ are ind-objects (e.g., $A=SV$ or $A=\wedge V$, where $V$ is an ind-object of $\C$). In this case, $A^![i]$ will be pro-objects of $\C$. 
\end{remark} 

\section{Lie algebras in symmetric tensor categories}

\subsection{Operadic Lie algebras} 
Recall (\cite{LV}) that the Lie operad ${\bf Lie}=\oplus_{n\ge 1}{\bf Lie}_n$ 
is the linear operad over $\Bbb Z$
generated by a single antisymmetric element $b\in {\bf Lie}(2)$ (the bracket) 
with defining relation
$$
b\circ (b\otimes 1)\circ ({\rm Id}+(123)+(132))=0
$$
in  ${\bf Lie}(3)$ (the Jacobi identity). 
For a commutative ring $R$, we will denote ${\bf Lie}_n\otimes R$ by ${\bf Lie}_n(R)$. 
Let $b_n\in {\bf Lie}_n$ be defined recursively by $b_2=b$, 
$b_n=b\circ (b_{n-1}\otimes 1)$. 
It is well known that ${\bf Lie}_n$ is a free abelian group 
of rank $(n-1)!$ with basis $b_n\circ \sigma$, $\sigma\in S_{n-1}$, 
where $S_{n-1}\subset S_n$ is the stabilizer of $1\in [1,n]$. 
Thus the restriction of the $S_n$-representation 
${\bf Lie}_n$ to $S_{n-1}$ is the regular representation
of $S_{n-1}$.  

\begin{definition}\label{operadic} An operadic Lie algebra in a symmetric tensor category $\C$  
is an (ind-)object $L\in \C$ equipped with a structure of a ${\bf Lie}$-algebra. 
\end{definition} 

Here we use the adjective ``operadic" 
because Definition \ref{operadic}  
does not coincide with the usual one even if $\C={\rm Vec}$, in the case ${\rm char}(\kk)=2$. Namely, Definition \ref{operadic} only requires that $[x,y]=-[y,x]$, while the standard definition of a Lie algebra requires a stronger condition $[x,x]=0$. 

Likewise, if $\C={\rm Supervec}$ is the category of supervector spaces (i.e., in the case of Lie superalgebras),
Definition \ref{operadic} does not coincide with the standard one also if ${\rm char}(\kk)=3$. 
Namely, if $x\in L$ is an odd element, then the standard definition 
requires that $[[x,x],x]=0$, which in characteristic $3$ does not follow 
from the Jacobi identity 
$$
[[x,y],z]+[[y,z],x]+[[z,x],y]=0.
$$

On the other hand, in the category 
${\rm Supervec}$ Definition \ref{operadic} coincides with the usual one in any 
characteristic different from $2,3$. 

Note that any associative algebra is naturally an operadic Lie algebra 
with operation $b=\mu-\mu^{op}$, where $\mu$ is the multiplication
(since this formula defines a morphism from the Lie operad to the associative operad). 

\subsection{Free operadic Lie algebras} 
Let $V\in \C$. 
Define the {\it free operadic Lie algebra} 
${\rm FOLie}(V)$ by the formula ${\rm FOLie}(V)=\oplus_{n\ge 1}{\rm FOLie}_n(V)$, 
where  
$$
{\rm FOLie}_n(V)=(V^{\otimes n}\otimes {\bf Lie}_n)_{S_n},
$$
and the subscript $S_n$ means coinvariants with respect to $S_n$ (cf. \cite{Fr}). 
This has an obvious bracket making ${\rm FOLie}(V)$ an operadic Lie algebra. 
Moreover, ${\rm FOLie}(V)$ is generated in degree $1$.  

It is clear that ${\rm FOLie}(V)$ has a universal property: 
its Lie (i.e., bracket-preserving) homomorphisms to any operadic Lie 
algebra $L$ are in natural bijection with $\Hom_{\mathcal{C}}(V,L)$. 
In other words, the functor ${\rm FOLie}$ is left adjoint 
to the forgetful functor on the category of operadic Lie algebras in $\C$
(which forgets the bracket). 

In particular, we have a natural Lie algebra map 
$$
\phi^V: {\rm FOLie}(V)\to TV
$$ 
corresponding to the inclusion of the degree $1$ part $V\to TV$
(since the tensor algebra $TV$ is an associative algebra and therefore an operadic Lie algebra).
We have $\phi^V=\oplus_{n\ge 1}\phi_n^V$, where $\phi_n^V: {\rm FOLie}_n(V)\to V^{\otimes n}$. 
Let $E_n(V):={\rm Ker}\phi_n^V$, and $E(V)=\oplus_{n\ge 1}E_n(V)$. Note that 
$E(V)$ is an ideal in ${\rm FOLie}(V)$. Also, 
$E_n(V)=0$ if ${\rm char}(\kk)=0$ or ${\rm char}(\kk)=p>n$ (in particular, $E_1(V)=0$), since this is so in the category of 
Schur functors.   

\begin{lemma}\label{flee} Let $A\in \C$ be an associative algebra, 
and 
$$
\psi: {\rm FOLie}(V)\to A
$$
a Lie homomorphism. Then $\psi|_{E(V)}=0$. 
\end{lemma} 

\begin{proof} The homomorphism $\psi$ factors through $TV$, so 
we may assume that $\psi=\phi^V$, which implies the statement. 
\end{proof} 

\subsection{Computation of $E_n(V)$}\label{compenv}  
Let us give a more explicit realization of 
$E_n(V)$. Recall (\cite{Re}, Theorem 8.16, formula (8.4.2), see also 8.6.5)
that for any $n\ge 2$,  ${\bf Lie}_n=\kk S_n\cdot \theta_n$ 
as an $S_n$-module, where $\theta_n$ is the Dynkin element
$$
\theta_n=({\rm Id}-(21))({\rm Id}-(321))...({\rm Id}-(n...1)),
$$
which satisfies $\theta_n^2=n\theta_n$ 
(the characteristic zero assumption of \cite{Re}, Subsection 8.4
is not needed for this statement). 
Thus, 
$$
{\bf Lie}_n=\kk S_n/{\rm Ann}_l(\theta_n),
$$
where 
${\rm Ann}_l(y):=\lbrace x\in \kk S_n: xy=0\rbrace$ 
is the left annihilator of $y$.
Hence, 
$$
{\rm FOLie}_n(V)=V^{\otimes n}/{\rm Ann}_r(\theta_n^\vee)V^{\otimes n},
$$
where 
$$
\theta_n^\vee=S(\theta_n)=({\rm Id}-(1...n))...({\rm Id}-(123))({\rm Id}-(12))
$$
(where $S$ is the antipode), and 
${\rm Ann}_r(y):=\lbrace x\in \kk S_n: yx=0\rbrace$ 
is the right annihilator of $y$.
Hence, we obtain

\begin{lemma}\label{env}
The map $\phi_n^V: {\rm, FOLie}_n(V)\to V^{\otimes n}$ 
is given by multiplication by $\theta_n^\vee$ on 
$V^{\otimes n}/{\rm Ann}_r(\theta_n^\vee)V^{\otimes n}$.
Thus 
$$
E_n(V)={\rm Ker}(\theta_n^\vee: V^{\otimes n}/{\rm Ann}_r(\theta_n^\vee)V^{\otimes n}\to V^{\otimes n}). 
$$
\end{lemma} 

\begin{corollary}\label{chara}
If $n$ is not divisible by $p={\rm char}(\kk)$, then $E_n(V)=0$. 
\end{corollary} 

\begin{proof} If $n$ is not divisible by $p$ then $\frac{\theta_n^\vee}{n}$ is an idempotent,
so $\theta_n^\vee$ acts on $V^{\otimes n}/{\rm Ann}_r(\theta_n^\vee)V^{\otimes n}$ 
by multiplication by $n$, and hence defines an injection 
$V^{\otimes n}/{\rm Ann}_r(\theta_n^\vee)V^{\otimes n}\to V^{\otimes n}$.  
\end{proof} 

\subsection{Explicit form of $E_n(V)$ for small $n$}\label{small}
Let us now compute $E_2(V)$. We only have to consider  the case ${\rm char}(\kk)=2$. 
In this case we have ${\rm FOLie}_2(V)=S^2V$, 
while its image in $TV$ is $\wedge^2V$
(note that $\wedge^2V=\Lambda^2 V$ in any symmetric tensor category). 
Thus $E_2(V)=V^{(1)}$, the Frobenius twist of $V$, i.e., the kernel of the natural morphism $S^2V\to \wedge^2V$ (see \cite{O}, Section 3). E.g., if $V$ is a usual vector space, then $V^{(1)}$ is the usual Frobenius twist of $V$.  
 
Let us now compute $E_3(V)$. By Corollary \ref{chara}, we only need to consider characteristic $3$. 
In this case, it is easy to see 
that ${\rm Ann}_r(\theta_3^\vee)$ is 4-dimensional and 
equals 
$$
{\rm Ann}_r(\theta_3^\vee)=({\rm Id}+(12))\cdot \kk S_3\oplus \kk {\bold e}_-,
$$ 
where 
${\bold e}_-=\sum_{\sigma\in S_3} {\rm sign}(\sigma)\sigma$. 
The quotient of $V^{\otimes 3}$ by 
$$
({\rm Id}+(12))\cdot \kk S_3\cdot V^{\otimes 3}=S^2V\otimes V
$$ 
is $\Lambda^2V\otimes V$, so 
$$
V^{\otimes 3}/{\rm Ann}_r(\theta_3^\vee)V^{\otimes 3}=\Lambda^2V\otimes V/{\bold e}_-V^{\otimes 3}.
$$
Now, $\theta_3^\vee=2({\rm Id}-c)$ on $\Lambda^2V\otimes V$, where $c:=(123)$
so 
$$
{\rm Ker}(\theta_3^\vee|_{\Lambda^2V\otimes V/{\bold e}_-V^{\otimes 3}})
={\rm Ker}(({\rm Id}-c)|_{\Lambda^2V\otimes V/{\bold e}_-V^{\otimes 3}})\subset 
{\rm Ker}(({\rm Id}-c)|_{V^{\otimes 3}/{\bold e}_-V^{\otimes 3}}).
$$
The permutation $c$ maps $\Lambda^2V\otimes V$ onto $V\otimes \Lambda^2V$ 
inside $V^{\otimes 3}$, while it acts trivially on 
${\rm Ker}(({\rm Id}-c)|_{\Lambda^2V\otimes V/{\bold e}_-V^{\otimes 3}})$,
hence 
$$
{\rm Ker}(({\rm Id}-c)|_{\Lambda^2V\otimes V/{\bold e}_-V^{\otimes 3}})=
{\rm Ker}(({\rm Id}-c)|_{V\otimes \Lambda^2V/{\bold e}_-V^{\otimes 3}})
$$
inside ${\rm Ker}(({\rm Id}-c)|_{V^{\otimes 3}/{\bold e}_-V^{\otimes 3}})$. 
Since $\Lambda^2V\otimes V\cap V\otimes \Lambda^2V=\Lambda^3V$, this implies that
$$
{\rm Ker}(({\rm Id}-c)|_{\Lambda^2V\otimes V/{\bold e}_-V^{\otimes 3}})=
{\rm Ker}(({\rm Id}-c)|_{\Lambda^3V/{\bold e}_-V^{\otimes 3}}).
$$
But $c$ acts trivially on $\Lambda^3V$, so 
$$
{\rm Ker}(({\rm Id}-c)|_{\Lambda^2V\otimes V/{\bold e}_-V^{\otimes 3}})=\Lambda^3V/{\bold e}_-V^{\otimes 3}.
$$
Thus, by Lemma \ref{env}
\begin{equation}\label{e3v}
E_3(V)=\Lambda^3V/{\bold e}_-V^{\otimes 3}.
\end{equation}

Let us now compute $E_3(V)$ directly in the case when $\C={\rm Supervec}$. 
Choosing a basis of $V$, we see that ${\rm FOLie}_3(V)$ 
is a direct sum of spaces of two kinds: 
$(X\otimes {\bf Lie}_3)_{S_3}$, where $X$ is the 3-dimensional space spanned by  
the permutations of $v\otimes x\otimes x$, where $v,x$ are either odd or even, 
and $(Y\otimes {\bf Lie}_3)_{S_3}$, where $Y$ is spanned by $y\otimes y\otimes y$, where $y$ is either odd or even. 

In the first case, $X=\kk S_3\otimes_{\kk S_2}\kk_\pm$, 
where $\kk_+$ is the trivial representation, $\kk_-$ the sign representation, and 
the sign depends on the parity of $x$. Thus, $(X\otimes {\bf Lie}_3)_{S_3}=
({\bf Lie}_3\otimes \kk_{\pm})^{S_2}$.
Recall that ${\bf Lie}_3(\kk)$ is the regular representation of $S_2$, 
so we get that $(X\otimes {\bf Lie}_3)_{S_3}$ is $1$-dimensional, 
spanned by $[[v,x],x]$. This element maps to a nonzero element 
of $V^{\otimes 3}$. Thus, this case does not contribute to $E_3(V)$. 

In the second case, $(Y\otimes {\bf Lie}_3)_{S_3}=({\bf Lie}_3\otimes \kk_\pm)_{S_3}$,
where the sign is $+$ if $y$ is even and $-$ if $y$ is odd.  The $S_3$-representation 
${\bf Lie}_3(\kk)$ is a nontrivial extension 
of the sign representation $\kk_-$ by the trivial representation $\kk_+$:
\begin{equation}\label{nontex}
0\to \kk_+\to {\bf Lie}_3(\kk)\to \kk_-\to 0. 
\end{equation}  
Thus, $(Y\otimes {\bf Lie}_3)_{S_3}$ is zero if $y$ is even and 
1-dimensional spanned by $[[y,y],y]$ if $y$ is odd. In the latter case, 
the basis vector maps to zero in $V^{\otimes 3}$, hence contributes 
to $E_3(V)$. 

Altogether, we see that $E_3(V)=V_{\rm odd}^{(1)}$, 
the Frobenius twist of the odd part of $V$ (so, it is 
a purely odd supervector space). Obviously, this coincides with \eqref{e3v}.  

\subsection{Additivity of $E_p(V)$ in characteristic $p$}

Let ${\rm char}(\kk)=p$. 

\begin{proposition}\label{addit} 
We have $E_p(V\oplus W)=E_p(V)\oplus E_p(W)$. 
\end{proposition} 

\begin{proof} Let ${\bf Assoc}=\oplus_{n\ge 1}{\bf Assoc}_n$ 
be the associative operad. Thus, ${\bf Assoc}_n=\Bbb ZS_n$. 
We have $E_p(V\oplus W)=E_p(V)\oplus E_p(W)\oplus \oplus_{i=1}^{p-1}E_{p,i}(V,W)$, where 
$E_{p,i}(V,W)$ is the kernel of the natural map 
$$
({\rm Ind}_{S_i\times S_{p-i}}^{S_p}(V^{\otimes i}\otimes W^{p-i})\otimes {\bf Lie}_p)_{S_p}\to 
({\rm Ind}_{S_i\times S_{p-i}}^{S_p}(V^{\otimes i}\otimes W^{p-i})\otimes {\bf Assoc}_p)_{S_p}. 
$$
Using Frobenius reciprocity, this can be rewritten as a map 
$$
\kappa: (V^{\otimes i}\otimes W^{p-i}\otimes {\bf Lie}_p)_{S_i\times S_{p-i}}\to
(V^{\otimes i}\otimes W^{p-i}\otimes {\bf Assoc}_p)_{S_i\times S_{p-i}}.
$$
But the natural map ${\bf Lie}_n\to {\bf Assoc}_n$ is injective, while 
the order of the group $S_i\times S_{p-i}$ is prime to $p$, hence the functor of coinvariants 
under this group is exact. Thus, $\kappa$ is injective, and $E_{p,i}(V,W)=0$, as desired. 
\end{proof} 

\subsection{Lie algebras} 
Now let $L$ be an operadic Lie algebra in $\C$. Then we have a natural  
morphism $\beta^L: {\rm FOLie}(L)\to L$. 

\begin{definition}
An operadic Lie algebra $L$ is called a Lie algebra if 
$\beta^L|_{E(L)}=0$. 
\end{definition} 

\begin{proposition}\label{anyasso}
Any associative algebra or its subobject closed under bracket is a Lie algebra.  
\end{proposition} 

\begin{proof}
Suppose $A\in \C$ is an associative algebra, and $L\subset A$ is closed under bracket. 
Let $i: L\to A$ be the inclusion map. Then $\theta=i\circ \beta^L: {\rm FOLie}(L)\to A$
is a Lie algebra map, and it suffices to check that $\theta|_{E(L)}=0$. 
But this is guaranteed by Lemma \ref{flee}.
\end{proof} 

\begin{example} 
1. Let $A$ be any algebraic structure in a symmetric tensor category $\C$,
i.e., a commutative algebra, associative algebra, Hopf algebra, etc. More precisely, 
let ${\mathcal{P}}$ be a PROP over $\kk$ (cf. \cite{M}), and $A$ be a ${\mathcal{P}}$-module in $\C$. 
Let $L:={\underline{\rm Der}}(A)\subset \underline{\rm Hom}(A,A)$ 
be the Lie subalgebra of {\it internal derivations} of $A$, i.e., internal endomorphisms 
that preserve the operations $A^{\otimes n}\to A^{\otimes m}$ coming from ${\mathcal{P}}(n,m)$
for all $n,m$. Then by Proposition \ref{anyasso}, $L$ is a Lie algebra.

2. Let $G$ be an affine group scheme in $\C$, i.e., we have a commutative Hopf (ind-)algebra $H=O(G)$ in $\C$. 
Let ${\mathfrak{m}}$ be the augmentation ideal in $H$. As in \cite{E}, Subsection 2.2, let ${\mathfrak{g}}={\rm Lie}(G):=
({\mathfrak{m}}/{\mathfrak{m}}^2)^*$ (in general, a pro-object of $\C$). 
Then by Proposition \ref{anyasso}, ${\mathfrak{g}}$ is a Lie algebra. 
Indeed, ${\mathfrak{g}}$ is contained in the algebra $H^*$ (as the kernel of the augmentation on $({\mathfrak{m}}^2)^\perp$) 
and is closed under bracket. One can also interpret ${\mathfrak{g}}$ as the Lie algebra of left-invariant derivations 
of the commutative algebra $O(G)$.  
\end{example}

\begin{corollary}\label{extendsuni} (i) The quotient ${\rm FLie}(V):={\rm FOLie}(V)/E(V)$
is a Lie algebra. 

(ii) If $L\in \C$ is a Lie algebra then any $\C$-morphism 
$V\to L$ extends uniquely to a Lie algebra morphism ${\rm FLie}(V)\to L$. 
In other words, the functor ${\rm FLie}$ is left adjoint 
to the forgetful functor on the category of Lie algebras in $\C$.
\end{corollary} 

\begin{proof} (i) By definition, ${\rm FLie}(V)$ is a subobject of $TV$ 
closed under bracket, which implies (i) by Proposition \ref{anyasso}. 

(ii) Given a $\C$-morphism 
$\eta: V\to L$, use the universal property of ${\rm FOLie}(V)$ 
to extend it to a morphism 
$\xi: {\rm FOLie}(V)\to L$. Clearly, $\xi=\beta^L\circ {\rm FOLie}(\eta)$.
Also $\phi^L\circ {\rm FOLie}(\eta)|_{E(V)}=T(\eta)\circ \phi^V|_{E(V)}=0$. 
Hence, ${\rm FOLie}(\eta)$ maps $E(V)$ to $E(L)$. 
Thus, $\xi|_{E(V)}=0$, which implies (ii).   
\end{proof} 

\subsection{The PBW theorem in the Koszul case}

Let $L\in \C$ be an operadic Lie algebra. Define the {\it universal enveloping algebra} 
$U(L)$ to be the quotient of the tensor algebra $TL$ by the two-sided 
ideal generated by the image of the map ${\rm Id}-(12)-b$ from 
$L\otimes L$ to $TL$, where $b: L\otimes L\to L$ is the commutator. 
Thus the natural map $L\to U(L)$ is bracket-preserving. 

As classically, the algebra $U(L)$ is determined by its universal property: 
for any associative algebra $A$ in ${\mathcal{C}}$, 
homomorphisms of associative algebras $U(L)\to A$ 
bijectively correspond to Lie homomorphisms $L\to A$. 
In particular, by Corollary \ref{extendsuni}, we have canonical isomorphisms 
$U({\rm FOLie}(V))\cong U({\rm FLie}(V))\cong TV$. 

Recall that $SL$ is the quotient of $TL$ by the image of 
${\rm Id}-(12)$. Thus, $U(L)$ carries a natural ascending filtration (defined by the condition $\deg(L)=1$), and we have a natural surjective algebra morphism 
$$
\eta: SL\to {\rm gr}U(L).
$$ 

\begin{definition} We say that $L$ is PBW
if $\eta$ is an isomorphism (i.e., the PBW theorem holds for $L$). 
\end{definition} 

Note that if $L$ is PBW then the natural map $L\to U(L)$ is injective, so 
by  Proposition \ref{anyasso}, $L$ is a Lie algebra. 

\begin{theorem}\label{PBWt} Suppose $L$ is a Koszul object of $\C$. 
Then the following conditions are equivalent. 

(1) $L$ is PBW;

(2) $L$ is a Lie algebra; 

(3) $\beta^L|_{E_n(L)}=0$ for $n=2$ and $n=3$. 

In particular, if ${\rm char}(\kk)$ is $0$ or $p\ge 5$, 
$L$ is always a Lie algebra and satisfies the PBW theorem. 
\end{theorem} 

\begin{proof} We have just seen that (1) implies (2). 
It is clear that (2) implies (3), so we just need to show that (3) implies (1).

To do so, let $\widehat{L}:=L\oplus \bold 1$, and introduce the 
formal universal enveloping algebra $U_\hbar(L)$ over $\kk[[\hbar]]$ 
to be the quotient of $T\widehat{L}[[\hbar]]$
by the ideal $I$ generated by the sum $R$ of the images of the morphisms
$\alpha: L\otimes \bold 1\to T\widehat{L}[[\hbar]]$ and $\gamma: 
L\otimes L\to T\widehat{L}[[\hbar]]$ given by 
$$
\alpha={\rm Id}-(12),\ \gamma={\rm Id}-(12)-\hbar r_L\circ b,
$$
where $r_L: L\to L\otimes \bold 1$ is the unit morphism. 
Clearly, $U_\hbar(L)/(\hbar)=S(L\oplus \bold 1)$, so 
$U_\hbar(L)$ is a quadratic deformation of $S(L\oplus \bold 1)$, and 
the PBW property of $L$ is equivalent to flatness of this deformation. 
So our job is to show that $U_\hbar(L)$ is flat. 
Also note that the object $L\oplus \bold 1$ is Koszul, since so is $L$. 

Flatness in degree $1$ is obvious. 
If $\beta^L|_{E_2(L)}=0$ then $\gamma$ factors through 
$\wedge^2L=\Lambda^2V$, which implies that $U_\hbar(L)$ is flat in degree $2$. 

Now consider degree $3$. We have $I[3]=L\otimes R+R\otimes L$.
Recall that $R/\hbar R=\wedge^2L=\Lambda^2L$. Thus, flatness of $U_\hbar(L)$ 
in degree $3$ is equivalent to the property that 
$(L\otimes R\cap R\otimes L)/(\hbar)$ is $L\otimes \Lambda^2L\cap \Lambda^2L\otimes L=\Lambda^3L$, 
the $S_3$-invariants in $L^{\otimes 3}$. It is well known (see \cite{PP}, Chapter 5) that the obstruction to this property 
is a certain ``Jacobi'' map $J: \Lambda^3L\to L$, which must vanish 
for the deformation to be flat modulo $\hbar^3$. 
It is easy to see that in all characteristics but $3$, this map is automatically zero, 
so let us consider characteristic $3$. 
In this case, note that by \eqref{nontex}, for any $V\in \C$, 
$(\Lambda^3V\otimes {\bf Lie}_3)_{S_3}=\Lambda^3V$. So 
we have a natural map $\tau^V: \Lambda^3V\to (V^{\otimes 3}\otimes {\bf Lie}_3)_{S_3}$ induced by the inclusion $\Lambda^3V\to V^{\otimes 3}$, 
and we have $J=\beta^L\circ \tau^L$. Since $\beta^L|_{E_3(L)}=0$, 
to establish flatness in degree $3$, 
it suffices to show that $\tau^V(\Lambda^3V)\subset E_3(V)$
for any $V\in \C$. 
 
To this end, note that $\phi^V\circ \tau^V=\theta_3^\vee=({\rm Id}-(123))({\rm Id}-(12))$. But the morphism 
${\rm Id}-(12)$ multiplies by $2=-1$ in $\Lambda^3V$, while ${\rm Id}-(123)$ acts on $\Lambda^3V$ by zero. 
Thus, $\phi^V\circ \tau^V=0$, hence $\tau^V$ lands in $E_3(V)$, as desired. So $U_\hbar(L)$ is flat in degree $3$. 

Now by the Koszul deformation principle (Theorem \ref{KdP}), $U_\hbar(L)$ is flat in all degrees.
This proves the theorem.   
\end{proof} 

This implies the usual PBW theorem for Lie algebras and superalgebras. Namely, we obtain the following 
(well known) result (cf. \cite{Mu}).  

\begin{corollary}\label{PBWt1} Suppose that $\C$ is super-Tannakian, i.e., admits a tensor functor $F: \C\to {\rm Supervec}$ (or $F: \C\to {\rm Vec}$ 
if ${\rm char}(\kk)=2$). Let $L\in \C$ be an operadic Lie algebra. Then the following conditions are equivalent: 

(1) $L$ is PBW;

(2) $L$ is a Lie algebra; 

(3) For any $x\in F(L)$, $[x,x]=0$ if ${\rm char}(\kk)=2$, and for any odd $x\in F(L)$, $[[x,x],x]=0$ 
if ${\rm char}(\kk)=3$. 

In particular, $L$ is always a Lie algebra and satisfies the PBW theorem 
in characteristic $0$ or $p\ge 5$.  
\end{corollary} 

\begin{proof}
Since $\C$ is super-Tannakian, $L$ is a Koszul object. 
Thus, Theorem \ref{PBWt} applies. Moreover, by the results of Subsection \ref{small}, 
condition (3) in Theorem \ref{PBWt} specializes to condition (3) of Corollary \ref{PBWt1}. 
This implies the result.   
\end{proof}  

\begin{remark} By virtue of Remark \ref{infii}, the results of this section also apply to 
infinite-dimensional Lie algebras, i.e., the case when $V,L$ are ind-objects of $\C$.
\end{remark}

We will see, however, that the PBW theorem for an operadic Lie algebra 
can fail in any characteristic $p\ge 5$ when the underlying object $L$ is not Koszul. 
Namely, this happens in the Verlinde category ${\rm Ver}_p$, and to ensure 
the PBW theorem, one needs to impose an equation of degree $p$, similar to the equations 
$[x,x]=0$ in characteristic $2$ and $[[x,x],x]=0$ for odd $x$ in characteristic $3$
considered above. This is discussed in subsequent sections.  

\section{Koszul complexes in the Verlinde category}

\subsection{Homology of Koszul complexes in the Verlinde category}
Let us now study Koszul complexes in the Verlinde category ${\rm Ver}_p$.
Recall (see e.g. \cite{O}, Subsection 3.2 and \cite{EOV}, Subsection 2.1) that ${\rm Ver}_p$ has simple objects $L_1={\bold 1}, L_2,...,L_{p-1}$, which are self-dual. 
The object $L_{p-1}$ is invertible and for $p\ge 3$ generates a copy of ${\rm Supervec}$, so we will denote it by $\bold 1_-$. 
Also, the subcategory ${\rm Ver}_p^+$ spanned by $L_i$ with odd $i$ is a tensor subcategory of ${\rm Ver}_p$, 
and ${\rm Ver}_p={\rm Ver}_p^+\boxtimes {\rm Supervec}$ for $p\ge 3$. 

\begin{proposition}\label{Koszulhom} 
The object $L_m\in {\rm Ver}_p$ is Koszul if and only if 
it is invertible, i.e., $m=1,p-1$. For $2\le m\le p-2$, the homology of 
$K^\bullet(L_m)$ is given by the formulas 
$$
H^0(K^\bullet(L_m))=\bold 1;\ H^m(K^\bullet(L_m))={\bold 1}_-^{\otimes m+1};
$$
$$
H^i(K^\bullet(L_m))=0\text{ for }i\ne 0,m.
$$
Moreover, the nontrivial homology in homological degree $m$ sits in diagonal degree 
$p$. 
\end{proposition} 

\begin{proof} First of all, if $m=1,p-1$ then $L_m\in {\rm Supervec}\subset {\rm Ver}_p$, hence Koszul. 
So it remains to consider the case $m\ne 1,p-1$. 
Recall (\cite{EOV}, Proposition 2.4) that 
$\Lambda L_m=\oplus_{j=0}^m\Lambda^jL_m$, 
and $SL_m=\oplus_{i=0}^{p-m}S^iL_m$.
This means that $K^\bullet(L_m)$ sits 
in diagonal degrees $0\le d\le p$. 
Since the Koszul complex is exact for Schur functors in diagonal degrees $<p$, 
nontrivial homology of $K^\bullet(L_m)$ 
may exist only in degree $p$. But 
$K^\bullet(L_m)[p]=K^m(L_m)[p]=S^{p-m}L_m\otimes \Lambda^mL_m$.
Now, it follows from \cite{EOV}, Proposition 6.1 that 
$S^{p-m}L_m={\bold 1}_-^{\otimes m+1}$, 
while $\Lambda^m L_m=\bold 1$. This implies the statement. 
\end{proof} 

\begin{corollary}\label{kun} Let $\C$ be a symmetric tensor category with 
a symmetric tensor functor $F: \C\to {\rm Ver}_p$. 
Then for any $V\in \C$ and $2\le m\le p-2$, we have 
$$
H^m(K^\bullet(V))={\rm Hom}(L_m,F(V))^{(1)}\otimes {\bold 1}_-^{\otimes m+1}. 
$$
Moreover, the algebra $H^\ast(K^\bullet(V))$ is an algebra in ${\rm Supervec}$ 
given by the formula 
$$
H^\ast(K^\bullet(V))=
$$
$$
S(\oplus_{\ell=1}^{\frac{p-3}{2}} {\rm Hom}(L_{2\ell},F(V))^{(1)}\otimes {\bold 1}_-)
\otimes \wedge(\oplus_{\ell=1}^{\frac{p-3}{2}} {\rm Hom}(L_{2\ell+1},F(V))^{(1)}\otimes {\bold 1}).
$$
Here ${\rm Hom}(L_m,F(V))^{(1)}$ sits in homological degree $m$ and diagonal degree $p$.
\end{corollary} 

\begin{proof} 
This follows from Proposition \ref{Koszulhom} and the K\"unneth formula. 
The Frobenius twist comes from the fact that the operator of multiplication 
by $\lambda\in \kk$ on $L_m$ acts on $S^{p-m}L_m\otimes \Lambda^m L_m$ by 
$\lambda^p$. 
\end{proof} 

\subsection{Almost Koszulity of symmetric and exterior algebras in ${\rm Ver}_p$}
Proposition \ref{Koszulhom} can be interpreted as the statement that 
the algebra $SL_m$ is almost Koszul in the sense of \cite{BBK}. 

\begin{definition} (a generalization of Definition 3.1, \cite{BBK}) Let
$A$ be a $\Bbb Z_+$-graded algebra in $\C$ with $A[0]=\bold 1$.
We call $A$ (left) almost Koszul if there exist integers
$r,s\ge 1$ such that $A[n]=0$ 
for all $n > r$, and there is a graded complex
$P_\bullet$: 
$$
 0\to P_s\to...\to P_1\to P_0\to 0
$$
of projective (=free) $A$-modules such that each
$P_i$ is generated by its component $P_i[i]$
of degree $i$, and the only non-zero homology is
$A[0]$ in degree $0$, and $W=A[r]\otimes P_s[s]$
in degree $r+s$. We say then that
$A$ is an $(r,s)$-Koszul algebra. 
\end{definition} 

\begin{proposition}\label{dualit} Suppose $r,s\ge 2$ and $A$ is an $(r,s)$-Koszul algebra. 
Then $A$ is quadratic, and $A^!$ is $(s,r)$-Koszul.  
\end{proposition} 

\begin{proof}
The proof is parallel to the proof of Proposition 3.11 of \cite{BBK}, which 
is the special case of this result for the category of vector spaces.  
\end{proof} 

\begin{theorem}\label{almkoz} The algebra $SL_m$ is $(p-m,m)$-Koszul, 
and the algebra $\wedge L_m$ is $(m,p-m)$-Koszul.  
\end{theorem}

\begin{proof} This follows immediately from Proposition \ref{Koszulhom} and Proposition \ref{dualit}. 
\end{proof} 

\subsection{The periodic Koszul complex and the Yoneda algebras of the symmetric and exterior algebras in ${\rm Ver}_p$}
Proposition \ref{Koszulhom} allows us to construct a periodic resolution 
of the augmentation module $\bold 1$ by free $SV$-modules 
for $V=L_m$ and $2\le m\le p-2$, similarly to \cite{BBK}. Namely, define 
the {\it periodic Koszul complex} $K^\bullet_{\rm per}(L_m)$  
to be a periodic extension of $K^\bullet(L_m)$ with period $m+1$.
Namely, $K^i_{\rm per}(L_m):=K^i(L_m)$ for $0\le i\le m$, and 
$$
K^{i+m+1}_{\rm per}(L_m)[r]=K^i_{\rm per}(L_m)[r+p]\otimes {\bold 1}_-^{\otimes m+1}
$$
for each $r>0$, where the differential $\partial: K^{i+1}_{\rm per}(L_m)\to K^i_{\rm per}(L_m)$
is periodic with period $m+1$ for $i\ge 0$, and $\partial_m={\rm Id}$
(note that this makes sense since $K^{m+1}_{\rm per}(L_m)=K^m_{\rm per}(L_m)=S^{p-m}L_m\otimes \Lambda^mL_m={\bold 1}_-^{\otimes m+1}$). 
Here the numbers in square brackets denote the diagonal degrees.  
    
By definition, the periodic Koszul complex is exact. 

\begin{corollary} \label{coro} 
For $2\le m\le p-2$ we have 
$$
\underline{\rm Ext}_{SL_m}^\bullet(\bold 1,\bold 1)=\wedge^\bullet (L_m\oplus {\bold 1}_-)=\wedge^\bullet L_m\otimes_- \wedge^\bullet {\bold 1}_-
$$
if $m$ is even, and 
$$
\underline{\rm Ext}_{SL_m}^\bullet(\bold 1,\bold 1)=\wedge^\bullet L_m\otimes  S^\bullet{\bold 1}
$$
if $m$ is odd, where the additional generator $\bold 1$ or $\bold 1_-$ 
sits in homological degree $m+1$ and diagonal degree $p$,
and $\otimes_-$ means the supertensor product (i.e., the generators anticommute).  
  
Similarly, we have 
$$
\underline{\rm Ext}_{\wedge L_m}^\bullet(\bold 1,\bold 1)=S^\bullet L_m\otimes_-  \wedge^\bullet {\bold 1}_-
$$
if $m$ is even, and 
$$
\underline{\rm Ext}_{\wedge L_m}^\bullet(\bold 1,\bold 1)=S^\bullet (L_m\oplus {\bold 1})=S^\bullet L_m\otimes S^\bullet {\bold 1}
$$
if $m$ is odd, where the additional generator $\bold 1$ or $\bold 1_-$ 
sits in homological degree $p-m+1$ and diagonal degree $p$.
\end{corollary}

\begin{proof} 
The additive structure follows by resolving the first copy of $\bold 1$ as an $SV$-module or $\wedge V^*$-module using 
the periodic Koszul complex $K^\bullet_{\rm per}(V)$ or its dual and computing internal Hom from 
this complex to $\bold 1$. The algebra structure is obtained
analogously to the paper \cite{Zh}, which computes the Yoneda algebra 
of a Frobenius almost Koszul algebra in the category of vector spaces. 
Namely, it is shown in \cite{Zh} that if $A$ is a Frobenius $(r,s)$-Koszul algebra over a field $\kk$
then the Yoneda algebra ${\rm Ext}_A(\kk,\kk)$ is the twisted polynomial algebra $A^![t]$, where 
$ta=\beta(a)t$ for $a\in A^!$ and $\beta$ is a certain Nakayama automorphism, and this statement can 
be generalized to the categorical setting, using that the algebra $SL_m$ and $\wedge L_m$ are Frobenius (in the sense that $A^*\cong A\otimes \chi$ 
as an $A$-module, where $\chi=\bold 1$ or $\chi={\bold 1}_-$).  
\end{proof} 

\begin{corollary}\label{coro1}
For any $V\in {\rm Ver}_p$, we have a bigraded algebra isomorphism
$$
\underline{\rm Ext}_{SV}^\bullet(\bold 1,\bold 1)=
$$
$$
\wedge^\bullet (V^*\oplus \bigoplus_{\ell=1}^{\frac{p-3}{2}} {\rm Hom}(L_{2\ell},F(V))^{(1)*}\otimes {\bold 1}_-)
\otimes S^\bullet(\bigoplus_{\ell=1}^{\frac{p-3}{2}} {\rm Hom}(L_{2\ell+1},F(V))^{(1)*}\otimes {\bold 1}),
$$
where $V^*$ sits in homological and diagonal degree $1$, and ${\rm Hom}(L_m,F(V))^{(1)*}$ sits in homological degree $m+1$ and 
diagonal degree $p$. 

Similarly, 
we have a bigraded algebra isomorphism
$$
\underline{\rm Ext}_{\wedge V^*}^\bullet(\bold 1,\bold 1)=
$$
$$
S^\bullet (V\oplus \bigoplus_{\ell=1}^{\frac{p-3}{2}} {\rm Hom}(L_{2\ell+1},F(V))^{(1)}\otimes {\bold 1})\otimes_-
\wedge^\bullet (\bigoplus_{\ell=1}^{\frac{p-3}{2}} {\rm Hom}(L_{2\ell},F(V))^{(1)}\otimes {\bold 1}_-),
$$
where $V$ sits in homological and diagonal degree $1$, and ${\rm Hom}(L_m,F(V))^{(1)}$ sits it homological degree $p-m+1$ and 
diagonal degree $p$.
\end{corollary}

\begin{proof} The statement is obtained from Corollary \ref{coro} by taking the tensor product and keeping track of the signs appropriately. 
\end{proof} 

\subsection{De Rham cohomology in ${\rm Ver}_p$}

In this subsection we compute the De Rham cohomology (defined in Subsection \ref{drc})
in the category ${\rm Ver}_p$.

\begin{proposition}\label{derham} For $1\le j,m\le p-1$, we have $H^j(DR^\bullet(L_m))[p]=0$ if $j\ne m$, and 
$H^m(DR^\bullet(L_m))[p]\cong {\bold 1}_-^{\otimes m+1}$. 
\end{proposition}

\begin{proof} If $m=1$, this is the Cartier isomorphism of Proposition \ref{cartier}. 
For $m=p-1$, this is the odd version of the Cartier isomorphism. For $2\le m\le p-2$, 
the De Rham complex lives in diagonal degrees $0\le d\le p$, so by Proposition \ref{divi}, 
it is exact everywhere except in degrees $0$ and $p$, and in diagonal degree $p$ it 
clearly has cohomology ${\bold 1}_-^{\otimes m+1}$ in homological 
degree $m$. 
\end{proof} 

\begin{corollary}\label{anyobj} Let $V\in {\rm Ver}_p$. Then for any $1\le m\le p-1$, 
we have a canonical isomorphism $H^m(DR^\bullet(V))[p]\cong {\rm Hom}(L_m,V)^{(1)}\otimes {\bold 1}_-^{\otimes m+1}$. 
\end{corollary} 

\begin{proof} This follows from Proposition \ref{derham} by taking a direct sum. 
\end{proof} 

Now let $\C$ be any symmetric tensor category over a field $\kk$ of characteristic $p$
which admits a symmetric tensor functor $F: \C\to {\rm Ver}_p$ (e.g. a fusion category, cf. \cite{O}, Theorem 1.5). 
Let $\C_{\rm st}$ be the largest super-Tannakian subcategory of $\C$, i.e., 
the full subcategory of all $X\in \C$ such that $F(X)$ is a direct sum of invertible objects. 
For instance, if $\C$ is finite then $\C_{\rm st}$ is the subcategory 
of objects with integer Frobenius-Perron dimension. 

\begin{corollary}\label{supertan} The functor $V\to H^m(DR^\bullet(V))[p]$ is twisted $\kk$-linear and lands 
in $\C_{\rm st}$. 
\end{corollary} 

\begin{proof} The twisted $\kk$-linearity follows from Proposition \ref{divi}(ii). The second statement 
follows from Corollary \ref{anyobj} by applying $F$. 
\end{proof} 
 
Now let $F_{\rm st}: \C_{\rm st}\to {\rm Supervec}$ be the restriction of $F$ to $\C_{\rm st}$. 

\begin{proposition}\label{formulaF} We have a canonical isomorphism
$$
F(V)\cong \oplus_{m=1}^{p-1}F_{\rm st}(H^m(DR^\bullet(V))[p])\otimes {\bold 1}_-^{\otimes m-1}\otimes L_m.
$$
\end{proposition} 

\begin{proof}
This follows from Corollary \ref{anyobj} and the fact that $F$ is exact and commutes with the functor $V\mapsto DR^\bullet(V)$.  
\end{proof} 

\begin{example} Suppose $\C_{\rm st}\subset {\rm Supervec}$ (e.g., $\C={\rm Ver}_p(G)$ for an adjoint simple algebraic 
group $G$ and sufficiently large $p$, see \cite{O}, 4.3.2, 4.3.3). Then 
Proposition \ref{formulaF} implies that 
$$
F(V)\cong \oplus_{m=1}^{p-1}H^m(DR^\bullet(V))[p]\otimes {\bold 1}_-^{\otimes m-1}\otimes L_m.
$$
\end{example} 

\subsection{Frobenius almost Koszul algebras over a field arising from tensor categories}

The relation between tensor categories and almost Koszul algebras used in this section 
has, in fact, been known for a long time. Namely, in \cite{BBK}, the main examples of 
almost Koszul algebras are preprojective algebras $\Pi$ of Dynkin quivers, 
which are Frobenius $(h-2,2)$-Koszul (even though they are not Koszul, unlike the case of non-Dynkin quivers, \cite{MOV}, Theorem 2.3, \cite{EE}). 
A proof of this fact based on the theory of tensor categories was essentially given in \cite{MOV}, Subsection 2.2,
where it is shown that $\Pi$ is the image under a tensor functor of the quantum symmetric algebra 
$S_qL_2$ in the Verlinde category ${\rm Ver}(q)={\rm Ver}(q,{\bold{SL}}(2))$ over $\Bbb C$ at $q=-e^{\pi i/h}$ (where $L_2$ is the tautological representation 
of ${\bold{SL}}_q(2)$). It remains to note that the algebra $S_qL_2$ is $(h-2,2)$-Koszul, 
with dual $\wedge_q L_2$, as easily follows by considering the $q$-deformed Koszul complex
$K_q^\bullet(L_2)=S_qL_2\otimes \wedge_qL_2$. 

This approach was extended in \cite{Co} to the case of ${\bold{SL}}_q(3)$, 
yielding examples of Frobenius $(h-3,3)$-Koszul algebras. Moreover, 
it can be extended to the case of ${\bold{SL}}_q(m)$ for any $m\ge 2$, yielding 
examples of Frobenius $(h-m,m)$-Koszul algebras over $\Bbb C$. Namely, consider the 
fusion category ${\mathcal{C}}={\rm Ver}(q,{\bold{SL}}(m))$ attached to the group ${\bold{SL}}(m)$
(\cite{BK}, Definition 3.3.19), with $h\ge m+2$, i.e., 
the quotient of the category of tilting representations of ${\bold{SL}}_q(m)$
by negligible morphisms. Let $V=V_{\omega_1}$ be the tautological $m$-dimensional representation 
of ${\bold{SL}}_q(m)$, and $S_qV=\oplus_{r=0}^{h-m}S_q^rV$ 
be the quantum symmetric algebra of $V$, where $S_q^rV:=V_{r\omega_1}$. 
Also let $\wedge_qV=\oplus_{r=0}^m\wedge_q^rV$ be the quantum exterior algebra of $V$, 
where $\wedge^r_qV:=V_{\omega_r}$ (and we agree that $\omega_m=0$). 
Then we have a $q$-deformed Koszul complex $K_q^\bullet(V)=S_qV\otimes \wedge_q^\bullet V$,
defined similarly to the classical case. It is easy to show that this complex is exact 
in all diagonal degrees except $0,h$, and its homology in degree $h$ is the invertible 
object $V_{(h-m)\omega_1}$ sitting in homological degree $h-m$. 
Namely, we have the Koszul complex $S_qV\otimes \wedge_q^\bullet V$
for generic $q$, which is shown to be exact similarly to the case $q=1$, 
and this argument still works if $q=-e^{\pi i/h}$ in degrees $<h$,
while in degree $h$ we just have a 1-dimensional space. 
This implies, similarly to the results of this section, that 
the algebra $S_qV$ is $(h-m,m)$ Koszul, with quadratic dual 
$\wedge_q V$ (which is then $(m,h-m)$-Koszul). 

These are almost Koszul algebras in the tensor category 
${\rm Ver}(q,{\bold{SL}}(m))$, but they can be used to construct 
Frobenius almost Koszul algebras over $\Bbb C$, using the method of \cite{MOV}. 
For this purpose, pick an indecomposable module 
category ${\mathcal{M}}$ over ${\mathcal{C}}$.
For $m=2$, such categories are known to be classified by 
ADET diagrams with Coxeter number $h$ (see \cite{KO,O1}),
and for small $m$ there are some classification results
by Ocneanu (\cite{Oc}), but for any $m$ we can, for example, take
the tensor category itself, ${\mathcal{M}}={\mathcal{C}}$. 
If ${\mathcal{M}}=R-{\rm mod}$ for a commutative semisimple algebra $R$
over $\Bbb C$, then ${\mathcal{M}}$ defines a tensor functor 
$F_{\mathcal{M}}: {\mathcal{C}}\to R-{\rm bimod}$ (cf. \cite{EGNO}, Chapter 7). 
Then the algebra $A:=F_{\mathcal{M}}(S_qV)$ is 
a Frobenius $(h-m,m)$-Koszul algebra over $\Bbb C$, with quadratic dual 
$A^!=F_{\mathcal{M}}(\wedge_qV)$, which is Frobenius $(m,h-m)$-Koszul; this 
follows from the above, since the functor $F_{\mathcal{M}}$ is monoidal and exact.   
This gives examples of Frobenius $(r,s)$-Koszul algebras over $\Bbb C$ 
for any $r,s\ge 2$. 

This construction makes sense modulo $p$ if $h,m$ are coprime to $p$,
giving examples of $(h-m,m)$-Koszul algebras over a field $\kk$ of any characteristic $p$ not dividing $h$
(as for $h$ coprime to $p$, either $m$ or $h-m$ is coprime to $p$). 
Also, it admits a reduction mod $p$ when $h=p$, 
with the category ${\rm Ver}(q,{\bold{SL}}(m))$ replaced by the symmetric tensor
category $\overline{\mathcal{C}}={\rm Ver}_p({\bold{SL}}(m))$
(see \cite{GK, GM}, and \cite{O}, 4.3.2), and the algebras $S_qV, \wedge_qV$ 
replaced with the usual symmetric and exterior algebras $SV$, $\wedge V$.  
Thus, given any semisimple indecomposable ${\mathcal{C}}$-module category ${\mathcal{M}}$,
we obtain a Frobenius $(p-m,m)$-Koszul algebra $A=F_{\mathcal{M}}(SV)$ over $\kk$, with quadratic dual 
$A^!=F_{\mathcal{M}}(\wedge V)$, which is Frobenius $(m,p-m)$-Koszul.
In particular, in this case there is a tensor functor $F: {\mathcal{C}}\to {\rm Ver}_p$ defined in \cite{O}, 4.3.3, 
such that $F(V)=L_m$. Thus, $F(SV)=SL_m$, $F(\wedge V)=\wedge L_m$, 
the almost Koszul algebras considered in this section. Moreover, 
for ${\mathcal{M}}={\rm Ver}_p$, we have 
$F_{\mathcal{M}}=\overline{F}_{\mathcal{M}} \circ F$, 
where 
$$
\overline{F}_{\mathcal{M}}: {\rm Ver}_p\to R-{\rm bimod}
$$
(the functor attached to ${\rm Ver}_p$ as a module over itself). 
So $A=\overline{F}_{\mathcal{M}}(SL_m)$ and $A^!=\overline{F}_{\mathcal{M}}(\wedge L_m)$. 

\section{The PBW theorem in ${\rm Ver}_p$} 

In this section we will prove the PBW theorem in ${\rm Ver}_p$. 
We may assume that ${\rm char}(\kk)=p\ge 5$. 

\subsection{Auxiliary results on the Lie operad and representations of $S_p$}
Let $p\ge 3$. Let $R_p$ be the $p-2$-dimensional irreducible representation 
of $S_p$ over $\kk$, on the space of functions 
on $[1,p]$ with zero sum of values modulo constants. 
This representation will play a key role in the whole story. 

\begin{lemma}\label{coinva} 
We have $(R_p\otimes {\bf Lie}_p)_{S_p}\cong \kk$. 
\end{lemma}  

\begin{proof} 
Note that ${\bf Lie}_p(\kk)_{S_p}=0$, since this is the degree 
$p$ component of the free Lie algebra in one generator, which is zero.  
Let $P_p$ be the permutation representation of $S_p$,
of dimension $p$, i.e., $P_p={\rm Ind}_{S_{p-1}}^{S_p}\kk.$ 
By Frobenius reciprocity,  $(P_p\otimes {\bf Lie}_p)_{S_p}={\bf Lie}_p(\kk)_{S_{p-1}}
=\kk$, as ${\bf Lie}_p(\kk)$ is the regular representation of $S_{p-1}$. 
Since $P_p$ is the extension of $R_p$ by $\kk$ on both sides, 
we have $\dim (R_p\otimes {\bf Lie}_p)_{S_p}\ge 1$. 

Also, we have a natural surjective homomorphism ${\rm Ind}_{S_{p-1}}^{S_p}{\bf Lie}_{p-1}\to {\bf Lie}_p(\kk)$, hence a surjective homomorphism 
$$
(R_p\otimes {\rm Ind}_{S_{p-1}}^{S_p}{\bf Lie}_{p-1})_{S_p}=
(R_p\otimes {\bf Lie}_{p-1})_{S_{p-1}}\to (R_p\otimes {\bf Lie}_p)_{S_p}.
$$
By Klyachko's theorem (see \cite{Re}, Corollary 8.7), for $n<p$ we have 
${\bf Lie}_n={\rm Ind}_{\Bbb Z/n}^{S_n}\chi_n$, 
where $\chi_n$ is a primitive (=injective) character
of $\Bbb Z/n$. Hence, 
$$
(R_p\otimes {\bf Lie}_{p-1})_{S_{p-1}}=
(R_p\otimes \chi_{p-1})_{\Bbb Z/(p-1)}. 
$$
But $R_p$ is the reflection representation of $S_{p-1}$, so contains every nontrivial character of $\Bbb Z/(p-1)$ with multiplicity $1$. 
Hence, 
$$
\dim (R_p\otimes \chi_{p-1})_{\Bbb Z/(p-1)}=\dim (R_p\otimes {\bf Lie}_{p-1})_{S_{p-1}}=1
$$
 and thus $\dim (R_p\otimes {\bf Lie}_p)_{S_p}\le 1$. 
Altogether we get 
$$
\dim (R_p\otimes {\bf Lie}_p)_{S_p}=1,
$$
 as desired. 
\end{proof} 

\begin{lemma}\label{vanlemma} 
The natural map 
$$
\zeta: (R_p\otimes{\bf Lie}_p)_{S_p}\to 
(R_p\otimes{\bf Assoc}_p)_{S_p}=R_p
$$ 
is zero. 
\end{lemma} 

\begin{proof}  
We have 
$$
\zeta(v\otimes \sigma \theta_p)=\theta_p^\vee\sigma^{-1} v
$$ 
for $\sigma\in S_p$ (see Subsection \ref{compenv}). 

But $\theta_p$ (and hence $\theta_p^\vee$) acts by zero on $R_p$. Indeed, suppose $v=(x_1,...,x_p)$ with $\sum x_i=0$ 
(considered modulo shifts of all coordinates by the same number). 
Then we can prove by reverse induction in $j$ that for any $2\le j\le p$
$$
(({\rm Id}-(j...1))...({\rm Id}-(p...1))v)_k=\sum_{i=0}^{p+1-j} (-1)^{i-1}\binom{p+1-j}{i}x_{i+k}
$$
for $k=1,...,j$, while the remaining coordinates are zero. 
In particular, taking $j=2$, we get $\theta_p v=(a,-a,0,...,0)$, where 
$$
a=\sum_{i=0}^{p-1} (-1)^{i-1}\binom{p-1}{i}x_{i+1}.
$$
Since we are in characteristic $p$, this implies that  $a=\sum_{i=0}^{p-1} x_{i+1}=0$,
as desired. 
\end{proof} 

\begin{lemma}\label{L2p} For $p\ge 3$ we have 
$$
L_2^{\otimes p}=\oplus_{\ell=1}^{\frac{p-3}{2}}Q_\ell\otimes L_{2\ell}\oplus R_p\otimes \bold 1_-,
$$
where $Q_\ell$ are projective $\kk S_p$-modules.   
\end{lemma} 

\begin{proof} The case $p=3$ is easy, so we may assume 
that $p\ge 5$. Since $p$ is odd, $L_2^{\otimes p}$ is a direct sum of $L_{2\ell}$, $1\le \ell\le \frac{p-1}{2}$,  
tensored with some multiplicity spaces $Q_\ell$, which are representations of $S_p$. 

For $r\le p-2$, $L_2^{\otimes r}$ is computed by the usual 
Clebsch-Gordan rule. So, ${\rm Hom}(L_{p-3},L_2^{\otimes p-2})=P_{p-3}$
(the reflection representation of $S_{p-2}$), while 
${\rm Hom}(L_{p-1},L_2^{\otimes p-2})=\kk$. Multiplying this 
by $L_2^{\otimes 2}={\bold 1}\oplus L_3$, we get $Q_{\frac{p-1}{2}}|_{S_{p-2}}=P_{p-3}+\kk=P_{p-2}$. 
Thus, $Q_{\frac{p-1}{2}}=R_p$. 

It remains to show that $Q_\ell$ is projective if $\ell\le \frac{p-3}{2}$. 
Since $\Bbb Z/p$ is a $p$-Sylow subgroup of $S_p$, 
it suffices to show that the restriction of $Q_\ell$ 
to $\Bbb Z/p$ is projective. This, in turn, is equivalent to 
this restriction being free or, equivalently, negligible. 
In other words, our job is to show that the Frobenius functor 
of $L_2$ in the sense of \cite{O}, Section 3 is 
${\rm Fr}(L_2)=\bold 1_-\boxtimes L_{p-2}\in {\rm Ver}_p\boxtimes {\rm Ver}_p$. 
But this is shown in \cite{O}, Example 3.8.    
\end{proof}

\subsection{The PBW theorem (preliminary version)}

Let $L$ be an operadic Lie algebra in ${\rm Ver}_p$. 
Suppose that $L$ is graded by nonnegative integers, 
$L=\oplus_{i\ge 0}L[i]$. We will call this grading {\it the external grading}. 

We first prove a preliminary version of the PBW theorem. 

\begin{theorem}\label{nol2}
Suppose that for some integer $r\ge 1$, we have \linebreak ${\rm Hom}(L_2,L[i])=0$ for $i<r$. 
Then the natural map $SL\to {\rm gr}U(L)$ is an isomorphism 
in external degrees $<pr$. 

In particular, if $L$ is any operadic Lie algebra in ${\rm Ver}_p$ 
with \linebreak ${\rm Hom}(L_2,L)=0$ (e.g., $L\in {\rm Ver}_p^+$) 
then $L$ is PBW and hence a Lie algebra.  
\end{theorem} 

\begin{proof}
As in the proof of Theorem \ref{PBWt}, consider the quadratic algebra 
$A:=U_\hbar(L)$. Its quadratic dual $A^!$ is the semidirect product 
of $\wedge L^*[[\hbar]]$ with the algebra generated by an odd derivation 
$d$ such that $d|_{L^*}=\hbar[,]^*: L^*\to \wedge^2L^*$, and it is a flat deformation of 
$B_0=\wedge (L^*\oplus \bold 1)$. Thus, $A=\underline{\rm Ext}_B(\bold 1,\bold 1)_{\rm diagonal}$ (see \cite{PP}, Proposition 3.1). So the associated graded of $A$ 
under the $\hbar$-adic filtration is computed by the spectral sequence attached to the deformation, starting with $A_0=A/\hbar A=S(L\oplus \bold 1)=\underline{\rm Ext}_{B_0}(\bold 1,\bold 1)$. 
 
The relevant differentials of this spectral sequence act from degree $(i-1,i)$ to degree 
$(i,i)$, where the first coordinate is the homological degree. Thus, 
nontrivial differentials can exist only for those $i$ for which 
$\underline{\rm Ext}_{B_0}^{i-1}(\bold 1,\bold 1)[i]\ne 0.$ 
By  Corollary \ref{coro1}, this can happen only if 
$\Hom(L_2,L)\ne 0$, and the smallest 
external degree where it can potentially 
happen is $pr$. This proves the first statement. 

The second statement follows from the first one 
by taking the trivial grading (i.e., 
the whole $L$ sits in degree zero) and considering all $r\ge 1$.
\end{proof} 

\begin{corollary}\label{Epcomp} 
(i) $E_p(L_m)=0$ if $m\ne 2$;

(ii) $E_p(L_2)=\bold 1_-$; 

(iii) For $V\in {\rm Ver}_p$, $E_p(V)={\rm Hom}(L_2,V)^{(1)}\otimes \bold 1_-=H^2(K^\bullet(V))=\underline{\rm Ext}^{p-1}_{\wedge V^*}(\bold 1,\bold 1)[p]$. 
\end{corollary} 

\begin{proof}
(i) Let $L={\rm FOLie}(L_m)$. Then $L$ has an external grading by positive integers, 
and does not contain $L_2$ in degree $1$. Hence, taking $r=2$ in Theorem \ref{nol2}, we see that 
the natural map $\eta: SL\to {\rm gr}U(L)$ is an isomorphism below degree $2p$. 
In particular, this is so in degree $p$, which implies that $E_p(L_m)=0$.   

(ii) By Lemma \ref{L2p}, $L_2^{\otimes p}=R_p\otimes \bold 1_-\oplus Q$, where 
$Q$ is projective over $\kk S_p$. So by Lemma \ref{coinva}, 
$E_p(L_2)$ is either $0$ or $\bold 1_-$. 
However, it follows from Lemma \ref{vanlemma}
that $E_p(L_2)\ne 0$. Thus, $E_p(L_2)=\bold 1_-$.

(iii) This follows from (i),(ii) and Proposition \ref{addit}, Corollary \ref{kun}, and Corollary \ref{coro1}. 
In more detail, we have a natural (split) inclusion 
$$
\iota: {\rm Hom}(L_2,V)\otimes L_2\to V,
$$ 
hence an inclusion 
$$
\iota_p: ({\rm Hom}(L_2,V)^{\otimes p}\otimes L_2^{\otimes p}\otimes {\bf Lie}_p)_{S_p}\to 
(V^{\otimes p}\otimes {\bf Lie}_p)_{S_p}.
$$
By Lemma \ref{L2p}, the multiplicity space of $\bold 1_-$ in $L_2^{\otimes p}$
is $R_p$, so by restriction we have a map 
$$
\xi: ({\rm Hom}(L_2,V)^{\otimes p}\otimes R_p\otimes {\bf Lie}_p)_{S_p}\otimes {\bold 1}_-\to (V^{\otimes p}\otimes {\bf Lie}_p)_{S_p}
$$
and further restricting to symmetric tensors, we get a map 
$$
\xi_{\rm sym}: {\Bbb S}^p{\rm Hom}(L_2,V)\otimes (R_p\otimes {\bf Lie}_p)_{S_p}\otimes {\bold 1}_-\to (V^{\otimes p}\otimes {\bf Lie}_p)_{S_p},
$$
i.e., using Lemma \ref{coinva},
$$
\xi_{\rm sym}: {\Bbb S}^p{\rm Hom}(L_2,V)\otimes {\bold 1}_-\to (V^{\otimes p}\otimes {\bf Lie}_p)_{S_p}.
$$
Now, we have a surjection ${\Bbb S^p}{\rm Hom}(L_2,V)\to {\rm Hom}(L_2,V)^{(1)}$, and 
by Proposition \ref{addit}, the map $\xi_{\rm sym}$ factors through
${\rm Hom}(L_2,V)^{(1)}\otimes \bold 1_-$, which embeds 
into $(V^{\otimes p}\otimes {\bf Lie}_p)_{S_p}$ 
as the kernel of $\phi^V_p$. Thus, $E_p(V)={\rm Hom}(L_2,V)^{(1)}\otimes \bold 1_-$, as stated. 
\end{proof} 

\subsection{The PBW theorem (general version)}
Now we proceed to study the PBW property in the case when $\Hom(L_2,L)\ne 0$. 
We will call the condition $\beta^L|_{E_p(L)}=0$ the {\it $p$-Jacobi identity.} 
Let us write this identity more explicitly for $p\ge 3$. 
Let $i_p: {\bold 1_-}\otimes R_p\to
L_2^{\otimes p}$ be the inclusion as a direct summand, see Lemma \ref{L2p}. 
Let $\bar \tau_p: \kk\to (R_p\otimes {\bf Lie}_p)_{S_p}$ 
be an isomorphism provided by Lemma \ref{coinva}, and 
$$
\tau_p={\rm Id}\otimes \bar\tau_p: {\bold 1}_-\to {\bold 1}_-\otimes (R_p\otimes {\bf Lie}_p)_{S_p}. 
$$
Let $\gamma: {\rm Hom}(L_2,L)\to {\rm Hom}({\bold 1}_-,L)$ be the map defined by the formula
$$
\gamma_p(x):=\beta^L\circ (x^{\otimes p}\circ i_p\otimes {\rm Id}_{{\bf Lie}_p})\circ \tau_p.
$$
Then $\gamma_p$ is twisted $\kk$-linear in $x$ (i.e., $\gamma(x+y)=\gamma(x)+\gamma(y)$ and 
$\gamma(\lambda x)=\lambda^p\gamma(x)$ for $\lambda\in \kk$; in other words, $\gamma_p$ is linear on ${\rm Hom}(L_2,L)^{(1)}$), and 
the $p$-Jacobi identity is 
$$
\gamma_p(x)=0 
$$
for all $x\in {\rm Hom}(L_2,L)$.
It is easy to see that for $p=3$, this identity reduces to 
the identity $[[x,x],x]=0$ 
for odd $x$. 

Our main result is the following PBW theorem. 

\begin{theorem}\label{maint} Let $L\in {\rm Ver}_p$ be an operadic Lie algebra. 
The following conditions are equivalent: 

(1) $L$ is PBW;

(2) $L$ is a Lie algebra; 

(3) $L$ satisfies the $p$-Jacobi identity. 
\end{theorem} 

\begin{proof} We may assume that $p\ge 5$. The PBW property implies 
that the natural map $L\to U(L)$ is injective, so 
by  Proposition \ref{anyasso}, (1) implies (2).
It is clear that (2) implies (3), so we just need to show that (3) implies (1).

We retain the notation of the proof of Theorem \ref{nol2}. 
Let $M_0=\oplus_i M_0[i]$, where $M_0[i]:=
\underline{\rm Ext}_{B_0}^{i-1}(\bold 1,\bold 1)[i]$. 
Then $M_0$ is a module over 
$A_0=\underline{\rm Ext}_{B_0}(\bold 1,\bold 1)_{\rm diagonal}$,
and Corollary \ref{coro1} implies that $M_0$ is generated in degree $i=p$,
and $M_0[p]=\Hom(L_2,L)^{(1)}\otimes {\bold 1}_-$. 

Now, if the differentials $d_j: M_0\to A_0$ of the spectral sequence vanish 
for $1\le j<s$, then $d_s$ is an $A_0$-linear map $M_0^s\to A_0$, where $M_0^s$ is 
a quotient $A_0$-module of $M_0$. Moreover, we claim that $M_0^s[p]=M_0[p]$. 
Indeed, by Corollary \ref{coro1}, $\underline{\rm Ext}_{B_0}^{p-2}(\bold 1,\bold 1)[p]=\Hom(L_3,L)^{(1)}\otimes \bold 1$, so 
all differentials from degree $(p-2,p)$ to $(p-1,p)$ vanish (as $\Hom({\bold 1}, {\bold 1}_-)=0$). 
Thus, to show that $d_s=0$, it suffices to show that the morphism
$d_s[p]: M_0[p]\to A_0[p]$ is zero. 
 
Now note that for an operadic Lie algebra $L$ 
in {\it any} symmetric tensor category in characteristic $p$, the algebra $U_\hbar(L)$ is a flat deformation 
of $SL$ modulo $\hbar^{p-1}$. Indeed, using Dynkin's explicit form of the Campbell-Baker-Hausdorff 
formula (\cite{Re}, 3.5.4), we see that the primes appearing in the denominators of coefficients of terms involving $\le n-1$ commutators are $\le n$. 
Hence, the standard proof of the PBW theorem based on the Campbell-Baker-Hausdorff formula (cf. \cite{EGNO}, Exercise 9.9.7(viii)) works modulo $\hbar^{p-1}$, as claimed. This implies that for $1\le s<p-1$ we have $d_s=0$. 

Let us now compute $d_{p-1}$ in degree $p$. 
We have $d_{p-1}[p]: M_0[p]\to L\otimes \bold 1^{\otimes p-1}=L$. 
But by Corollary \ref{Epcomp}, 
${\rm Hom}(L_2,L)^{(1)}\otimes \bold 1_-=E_p(L)$, so 
$d_{p-1}[p]: E_p(L)\to L$. 
Moreover, since $d_{p-1}[p]$ characterizes the failure of $U_\hbar(L)$ to be a flat deformation modulo $\hbar^p$ 
in degree $p$, we have $d_{p-1}[p]=\beta^L|_{E_p(L)}$. 

Thus, if $L$ satisfies the $p$-Jacobi identity, 
then $d_{p-1}[p]=0$, hence $d_{p-1}=0$. 

Then the differentials $d_s$ 
for $s>p-1$ vanish in degree $p$
for degree reasons (as they land $S^iL\otimes \bold 1^{\otimes p-i}$ for $i\le 0$). 
This implies that the spectral sequence in question has no nontrivial differentials into $A_0$. Hence 
the PBW theorem holds for $L$, as desired. 
\end{proof} 

\begin{corollary}\label{anystc}
Theorem \ref{maint} holds in any symmetric tensor category $\C$ which admits a symmetric tensor functor 
$F: \C\to {\rm Ver}_p$. In particular, it holds for any symmetric fusion category.
\end{corollary} 

\begin{proof} The first statement is obtained easily by applying the functor $F$
and using Theorem \ref{maint}. The second statement follows from the first one and 
\cite{O}, Theorem 1.5. 
\end{proof} 

\begin{corollary}\label{hilser} In the setting of Corollary \ref{anystc}, we have the following 
identity in ${\rm Gr}(\C)[[t]]$, which allows one to compute recursively the class of ${\rm FLie}_n(V)$ 
in the Grothendieck ring ${\rm Gr}(\C)$: 
$$
\prod_{n\ge 1} (1-t^n)^{{\rm FLie}_n(V)}=1-tV.
$$
\end{corollary} 

\begin{proof} The proof is the same as in the classical case. Namely, 
we have isomorphisms of graded algebras 
$S{\rm FLie}(V)\cong {\rm gr}U({\rm FLie}(V))$ (given by Theorem \ref{maint}) and $U({\rm FLie}(V))\cong TV$ (as both sides satisfy the same universal property). Hence we have a graded isomorphism $S{\rm FLie}(V)\cong {\rm gr}(TV)$. The statement now follows
by taking the Hilbert series in ${\rm Gr}(\C)[[t]]$ and inverting both sides. 
\end{proof} 

\begin{question} Let $L$ be an operadic Lie algebra in any symmetric tensor category $\C$ over a field 
$\kk$ of characteristic $p>0$. Does condition (2) in Theorem \ref{maint} imply (1)? Does (3) imply (2)?  
\end{question} 

In Section 7, we show that (2) implies (1) in quasi-semisimple categories. 

\subsection{Operadic Lie algebras failing the PBW theorem}

Let \linebreak ${\rm FOLie}(V)_{\le n}$ denote the quotient of ${\rm FOLie}(V)$ by 
degrees $\ge n+1$.   

\begin{proposition}\label{fail} For any $n\ge p$, the operadic Lie algebra ${\rm FOLie}(L_2)_{\le n}$ fails 
the PBW theorem, so is not a Lie algebra (and in particular fails the $p$-Jacobi identity).  
\end{proposition} 

\begin{proof} By Corollary \ref{Epcomp}(ii), $E_p(L_2)\ne 0$, so 
the surjective map 
$$
{\rm FOLie}(L_2)_{\le n}\to {\rm FLie}(L_2)_{\le n}
$$ 
is not an isomorphism  in degree $p$. Hence the natural map 
$L_2\to {\rm FOLie}(L_2)_{\le n}$ (the isomorphism onto degree $1$) 
does not factor through ${\rm FLie}(L_2)$. By Corollary \ref{extendsuni}(ii),
this means that ${\rm FOLie}(L_2)_{\le n}$ is not a Lie algebra, so the result follows 
from Theorem \ref{maint}. 
\end{proof} 

\begin{example} Let $p=5$, and 
let us prove Proposition \ref{fail} by a direct computation. 
 
Let $V=L_2\in {\rm Ver}_5$. Then $L_3=V_-:=V\otimes {\bold 1}_-$. 

Let $L={\rm FOLie}(V)_{\le 5}$, Then it is easy to see that we have $L[1]=V$, $L[2]=\wedge^2V=\bold 1$, 
$L[3]=[L[1],L[2]]=V$, $L[4]=L_3=V_-$.

We also claim that $L[5]=V\oplus \bold 1_-$. To see this, 
note that by Lemma \ref{L2p}, 
$V^{\otimes 5}=P_5\otimes V\oplus R_5\otimes \bold 1_-$, so the claim follows from Lemma \ref{coinva}. 

Now we claim that the natural map $\eta: SL\to {\rm gr}U(L)$ is not injective in (external) degree $5$. 
To see this, look at the the multiplicity of $\bold 1_-$. For 
${\rm gr}U(L)$ or, equivalently, $U(L)$, it equals
the corresponding multiplicity for $(TV)[5]=V^{\otimes 5}$, since $U(L)[m]=TV[m]$ for $m\le 5$. 
By Lemma \ref{L2p}, this multiplicity is equal to ${\rm dim} R_5=3$. 
On the other hand, $SL=\otimes_{i=1}^5 SL[i]$, and 
we get a copy of $\bold 1_-$ from $L[5]$, another copy from $L[1]\otimes L[4]$, 
another from $S^2L[1]\otimes L[3]$, and another from $S^3L[1]\otimes L[2]$, altogether $4$ copies. 

Thus, the PBW theorem fails for $L$. The same statement (with the same proof) applies 
to the Lie algebra 
$$
L':=L/V[5]=V\oplus \bold 1\oplus V\oplus V_-\oplus \bold 1_-,
$$
where $V[5]$ is the copy of $V$ in degree $5$. 
\end{example} 

\section{The PBW theorem in general symmetric tensor categories} 

\subsection{Differential operators} 

Let $\C$ be a symmetric tensor category over a field $\kk$ of any characteristic, and $V\in \C$. 
We will now define the algebra $D(V)$ of (crystalline) differential operators on $V$, which is an ind-algebra in $\C$. 

\begin{definition} The algebra $D(V)$ of differential operators on $V$ is the quotient of the tensor algebra 
$T(V\oplus V^*)$ by the ideal generated by $\Lambda^2V\oplus \Lambda^2V^*\oplus E$, where 
$E$ is the image of the morphism 
$$
P-{\rm Id}-{\rm ev}_V: V^*\otimes V\to T(V\oplus V^*),
$$
where $P$ is the permutation and ${\rm ev}_V: V^*\otimes V\to \bold 1$ is the evaluation map. 
\end{definition} 

Note that if $V$ is a finite dimensional vector space, this gives the usual algebra of (crystalline) differential 
operators on $V$, i.e., the Weyl algebra. 
 
The algebra $D(V)$ carries a filtration defined by ${\rm deg}(V)=1$, ${\rm deg}(V^*)=0$, and we have a natural surjective morphism 
$$
\eta: S(V\oplus V^*)\to {\rm gr}(D(V))
$$ 
(since the relations of $S(V\oplus V^*)$ are those of 
$D(V)$ without the lower-degree summand ${\rm ev}_V$). 
 
\begin{proposition}\label{PBWdiff}  (PBW theorem for differential operators) The morphism $\eta$ is an isomorphism.
\end{proposition} 

\begin{proof} Define an action of $D(V)$ on $S(V\oplus V^*)=SV\otimes SV^*$ as follows. Let $V^*$ act by $1\otimes \mu_{V^*}$, 
where  $\mu_{V^*}: V^*\otimes SV^*\to SV^*$ is the multiplication, and let $V$ act by $\mu_V\otimes 1+1\otimes \partial_V$,
where $\partial_V: V\otimes SV^*\to SV^*$ is the differentiation map (obtained by dualizing the multiplication map on $\Bbb SV$). 
It is easy to check that this defines an action of $T(V\oplus V^*)$ on $S(V\oplus V^*)$ which descends to an action of $D(V)$ 
(i.e., the relations of $D(V)$ are satisfied). Moreover, applying this action to the unit $\bold 1\subset S(V\oplus V^*)$, we obtain
a filtered $\C$-morphism $\zeta:  D(V)\to S(V\oplus V^*)$ such that ${\rm gr}(\zeta)\circ \eta={\rm Id}$. This implies that 
$\eta$ is injective, hence an isomorphism.  
\end{proof} 

\subsection{The PBW theorem for associative algebras} 

\begin{corollary}\label{assocal}  Let $L\in \C$ be an associative algebra. Then $L$ is PBW when regarded as a Lie algebra. 
\end{corollary}  

\begin{proof} First assume that $L$ is a unital algebra. Let 
$$
\delta: L\to L^*\otimes L
$$ 
be the dual of the multiplication map $m: L\otimes L\to L$. 
Since the object $L^*\otimes L$ sits canonically inside $D(L)$ (as linear vector fields), this defines a morphism $\delta: L\to D(L)$. 
It is easy to check that $\delta$ is, in fact, a Lie algebra homomorphism. Thus it gives rise to an associative algebra homomorphism $U(L)\to D(L)$, which we will also denote by $\delta$. 

Recall that we have a surjective morphism $\eta_L: SL\to {\rm gr}U(L)$, and our job is to show that it is injective. 
To this end, it suffices to show that the algebra morphism 
$\widetilde{\delta}:={\rm gr}(\delta)\circ \eta_L :SL\to {\rm gr}D(L)$  is injective. But by Proposition \ref{PBWdiff}, 
we have a canonical isomorphism ${\rm gr}D(L)\cong S(L\oplus L^*)=SL^*\otimes SL$, and upon this identification, 
$\widetilde{\delta}$ is induced by the morphism $\delta: L\to L^*\otimes L\subset S(L^*\oplus L)$. Let $\varepsilon: SL^*\to S\bold 1$ 
be the algebra map induced by the map $\varepsilon: L^*\to \bold 1$ dual to the unit map $\iota: \bold 1\to L$. 
Then the composition $(\varepsilon\otimes 1)\circ \widetilde{\delta}$ maps 
$SL$ to $S\bold 1\otimes SL=\kk[t]\otimes SL$, and it is easy to check 
that 
$$
|_{t=1}\circ (\varepsilon\otimes 1)\circ \widetilde{\delta}={\rm Id}_{SL}.
$$
This implies that $\widetilde{\delta}$ is injective, as desired. 

In general, the object $\hat L:=L\oplus \bold 1$ is naturally 
a unital algebra, and $U(\hat L)=U(L)\otimes S\bold 1$, 
which implies the required statement in full generality. 
\end{proof} 

\begin{remark} It is easy to show that Corollary \ref{assocal} applies to ind-algebras $L$, with a similar proof. 
The proof is modified in a standard way by considering appropriate inductive and projective limits. 
\end{remark} 

\subsection{The PBW theorem for a general Lie algebra}

\begin{definition} A symmetric tensor category $\C$ is called quasi-semisimple if for any 
injection $X\to Y$ in $\C$, the corresponding map $SX\to SY$ is injective. 
\end{definition} 

It is easy to see that any semisimple symmetric tensor category is quasi-semisimple, and 
any symmetric tensor category that admits a tensor functor into a quasi-semisimple (e.g., semisimple) 
symmetric tensor category is itself quasi-semisimple. Also note that examples of symmetric tensor categories which are not quasi-semisimple 
are known only in characteristic $2$, see \cite{V}, Subsection 2.3.
   
\begin{lemma}\label{subquot} 
(i) Let $L$ be an operadic Lie algebra in $\C$, and $I\subset L$ be an ideal. If $L$ is PBW then $L/I$ is PBW. 

(ii) Let $\C$ be quasi-semisimple, and $L_1\subset L_2$ be operadic Lie algebras in $\C$. If $L_2$ is PBW then so is $L_1$.
\end{lemma} 

\begin{proof} (i) We have ${\rm gr}U(L/I)={\rm gr}(U(L)/U(L)I)={\rm gr}U(L)/{\rm gr}(U(L)I)$. 
Since $I$ is a Lie ideal, we have ${\rm gr}(U(L)I)={\rm gr}U(L)\cdot I$. Since $L$ is PBW, 
we have ${\rm gr}U(L)=SL$. Thus, ${\rm gr}U(L/I)=SL/SL\cdot I=S(L/I)$, i.e., $L$ is PBW. 

(ii) We have surjective morphisms $\eta_i: SL_i\to {\rm gr}U(L_i)$, and $\eta_2$ is injective. 
Let $j: L_1\to L_2$ be the embedding, and denote the corresponding maps $SL_1\to SL_2$, $U(L_1)\to U(L_2)$ by $S(j),U(j)$ respectively. 
Then we have ${\rm gr}U(j)\circ \eta_1=\eta_2\circ S(j)$. But $S(j)$ is injective since $\C$ is quasi-semisimple. 
Hence $\eta_1$ is injective, as desired.  
\end{proof} 

\begin{proposition}\label{subal} Let $\C$ be quasi-semisimple, and $L\subset A$ be a Lie subalgebra of an associative algebra $A$. 
Then $L$ is PBW. 
\end{proposition}

\begin{proof} By Corollary \ref{assocal}, $A$ is PBW, so the result follows from Lemma \ref{subquot}(ii). 
\end{proof}

\begin{corollary}\label{freel} Let $\C$ be quasi-semisimple and $V\in \C$. Then the free Lie algebra ${\rm FLie}(V)$
is PBW. In particular, Corollary \ref{hilser} holds for $V$. 
\end{corollary} 

\begin{proof} This follows from Proposition \ref{subal}, since ${\rm FLie}(V)\subset TV$. 
\end{proof}   

\begin{theorem}\label{PBWquas} If $\C$ is quasi-semisimple then any Lie algebra $L$ in $\C$ is PBW.  
\end{theorem} 

\begin{proof} By definition, $L$ is a quotient of ${\rm FLie}(L)$, so the statement follows from Corollary \ref{freel} and Lemma \ref{subquot}(i). 
\end{proof} 

\subsection{Free restricted Lie algebras} 

Let $V\in \C$. The following definition is motivated by \cite{Fr}, Theorem 0.1. 

\begin{definition} The free restricted Lie algebra attached to $V$ is 
${\rm FRLie}(V)=\oplus_{n\ge 1}{\rm FRLie}_n(V)$, where 
$$
{\rm FRLie}_n(V)=(V^{\otimes n}\otimes {\bf Lie}_n)^{S_n}.
$$
\end{definition} 

This definition differs from the definition of ${\rm FOLie}(V)$ only in one aspect --- instead of $S_n$-coinvariants we take 
$S_n$-invariants. In particular, similarly to ${\rm FOLie}(V)$, ${\rm FRLie}(V)$ carries a natural grading-preserving bracket. 
Moreover, we have a natural Lie homomorphism $\pi_V: {\rm FOLie}(V)\to {\rm FRLie}(V)$ (the natural 
map from $S_n$-coinvariants to $S_n$-invariants; see \cite{Fr}, 1.1.20). It is an isomorphism in characteristic zero, but 
in general, it is neither surjective not injective. 

Note that we have a natural Lie homomorphism 
$$
\varphi^V: {\rm FRLie}(V)\to TV,
$$ 
induced by the inclusion of operads ${\bf Lie}\to {\bf Assoc}$. 
Moreover, since the functor of $S_n$-invariants is left exact, the homomorphism $\varphi^V$ is injective, i.e., ${\rm FRLie}(V)$ is a Lie algebra. 
Also, it is clear that $\varphi^V\circ \pi_V=\phi^V$, which implies that ${\rm Ker}\pi_V=E(V)$, so $\pi_V$ factors through ${\rm FLie}(V)$, which is a Lie subalgebra in 
${\rm FRLie}(V)$. 

\begin{corollary} Let $\C$ be quasi-semisimple. If $L$ is a Lie quotient of ${\rm FRLie}(V)$ 
then $L$ is PBW (in particular, a Lie algebra).
\end{corollary}

\begin{proof} This follows from Proposition \ref{subal} and Lemma \ref{subquot}(i), since 
$$
{\rm FRLie}(V)\subset TV.
$$ 
\end{proof}


\begin{thebibliography}{999999}
\bibitem[BK]{BK} B. Bakalov, A. Kirillov Jr., Lectures on tensor categories 
and and modular functors, AMS, 2001.

\bibitem[BBK]{BBK} S. Brenner, M.C.R. Butler, A. King: Periodic algebras which are almost Koszul, Algebras and Representation Theory 5 (4) 331---367.

\bibitem[Co]{Co} B. Cooper, Almost Koszul Duality and
Rational Conformal Field Theory, Ph.D. thesis, 2007, 
http://people.exeter.ac.uk/bc238/BCthesis.pdf

\bibitem[E]{E} P. Etingof, Representation theory in complex rank, II, to appear in Advances in Math., 
arXiv:1407.0373. 

\bibitem[EE]{EE} P. Etingof, C.-H. Eu, 
Koszulity and the Hilbert series of preprojective algebras, MRL, v.14, 2007, p. 589---596.

\bibitem[EGNO]{EGNO} P. Etingof, S. Gelaki, D. Nikshych, and V. Ostrik, 
Tensor categories, AMS, Providence, 2015. 

\bibitem[EHO]{EHO} P. Etingof, N. Harman, V. Ostrik, 
    $p$-adic dimensions in symmetric tensor categories in characteristic $p$,  arXiv:1510.04339.
    
\bibitem[EOV]{EOV} P. Etingof, V. Ostrik, S. Venkatesh, 
Computations in symmetric fusion categories in characteristic $p$,
arXiv:1512.02309.

\bibitem[Fr]{Fr} B. Fresse, On the homotopy of simplicial algebras over an operad, 
Transactions of the Amer. Math. Soc., 
Volume 352, Number 9, Pages 4113---4141, 2000. 

\bibitem[J]{J}  N. Jacobson, Lie Algebras, John Wiley and Sons, 1966.

\bibitem[GK]{GK} S. Gelfand and D. Kazhdan,
Examples of tensor categories,
Invent. Math.
109
(1992), no.
3, 595---617.

\bibitem[GM]{GM} G. Georgiev and O. Mathieu,
Fusion rings for modular representations of Chevalley groups,
Contemp. Math.
175
(1994), 89---100.

\bibitem[KO]{KO} A. Kirillov, V. Ostrik, 
On a q-analogue of the McKay correspondence and the ADE classification of $sl_2$ conformal field theories,
Advances in Mathematics 171 (2), 183---227.

\bibitem[LV]{LV} J. L. Loday, B. Valette, Algebraic Operads, Springer, 2012.  

\bibitem[M]{M} S. Mac Lane. Categorical algebra. Bulletin Amer. Math. Soc. 71(1965), 40Ð--106. 

\bibitem[Mu]{Mu} Ian M. Musson, Lie Superalgebras and Enveloping Algebras , Graduate Studies in Mathematics, 2012.

\bibitem[MOV]{MOV}
A. Malkin, V. Ostrik, M. Vybornov, Quiver varieties and Lusztig's algebra, Advances in Mathematics,
Volume 203, Issue 2, 2006, Pages 514---536.

\bibitem[O]{O} V. Ostrik, On symmetric fusion categories in positive characteristic, arXiv:1503.01492. 
 
\bibitem[O1]{O1} V. Ostrik, Module categories over representations of $SL_q(2)$ in the nonsemisimple case,
GAFA, Vol. 17 (2008), pp. 2005---2017. 

\bibitem[Oc]{Oc}  A. Ocneanu, The classification of subgroups of quantum SU(N). 
Quantum symmetries in theoretical physics and mathematics (Bariloche, 2000), 
133---159, Contemp. Math., 294, Amer. Math. Soc., Providence, RI, 2002.

\bibitem[PP]{PP} A. Polishchuk, L. Positselski, Quadratic Algebras. University Lecture Series 37, American Mathematical Society, Providence, RI, 2005.

\bibitem[V]{V} S. Venkatesh, Hilbert Basis Theorem and Finite Generation of Invariants in Symmetric Fusion Categories in Positive 
Characteristic, IMRN, doi: 10.1093/imrn/rnv305, 2015, 
arXiv:1507.05142. 

\bibitem[Re]{Re} C. Reutenauer, Free Lie algebras, London Mathematical Society Monographs. New Series, 7.
Oxford Science Publications. The Clarendon Press, Oxford University Press, New York, 1993.

\bibitem[Zh]{Zh} L. Zheng, Yoneda algebras of almost Koszul algebras, 
Proc. Indian Acad. Sci., v. 125, issue 4, 2015, pp. 477---485.  

\end{thebibliography}
\end{document}